\documentclass[11pt,reqno]{amsart}

\usepackage{amsmath,amssymb,amsthm,esint}
\usepackage{bm}
\usepackage{mathrsfs}
\usepackage[all]{xy}
\usepackage{booktabs}
\usepackage{fullpage}
\usepackage[pagebackref=true]{hyperref}
\usepackage{mathtools}
\usepackage[abbrev,msc-links]{amsrefs}
\usepackage[]{todonotes}
\hypersetup{
 hidelinks,
 bookmarksopen=true,
}

\newcommand{\itg}{\mathbb{Z}}
\newcommand{\rtn}{\mathbb{Q}}

\newcommand{\rl}{\mathbb{R}}
\newcommand{\cx}{\mathbb{C}}

\newcommand{\CB}{\mathcal{B}}

\newcommand{\CF}{\mathcal{F}}

\newcommand{\CH}{\mathcal{H}}

\newcommand{\CL}{\mathcal{L}}

\newcommand{\CO}{\mathcal{O}}

\newcommand{\CU}{\mathcal{U}}
\newcommand{\CV}{\mathcal{V}}

\newcommand{\CX}{\mathcal{X}}
\newcommand{\CY}{\mathcal{Y}}

\newcommand{\Ft}{\mathfrak{t}}

\newcommand{\ai}{\sqrt{-1}}

\newcommand{\inj}{\hookrightarrow}

\newcommand{\isom}{\stackrel{\sim}{\to}}

\newcommand{\kah}{K\"ahler }
\newcommand{\ke}{K\"ahler--Einstein }

\newcommand{\ddbar}{\partial \bar{\partial}}

\newcommand{\actson}{\curvearrowright}
\newcommand{\prj}{\mathbb{P}}

\theoremstyle{plain}
\newtheorem{theorem}{Theorem}[section]

\newtheorem{lemma}[theorem]{Lemma}
\newtheorem{proposition}[theorem]{Proposition}
\newtheorem{corollary}[theorem]{Corollary}
\theoremstyle{definition}
\newtheorem{definition}[theorem]{Definition}

\theoremstyle{definition}
\newtheorem{remark}[theorem]{Remark}

\begin{document}

\title{Anticanonically balanced metrics and the Hilbert--Mumford criterion for the $\delta_m$-invariant of Fujita--Odaka}
\author[Y.~Hashimoto]{Yoshinori Hashimoto}
\date{\today}
\address{Department of Mathematics, Osaka Metropolitan University, 3-3-138, Sugimoto, Sumiyoshi-ku, Osaka, 558-8585, Japan.}
\email{yhashimoto@omu.ac.jp}

\begin{abstract}
We prove that the stability condition for Fano manifolds defined by Saito--Takahashi, given in terms of the sum of the Ding invariant and the Chow weight, is equivalent to the existence of anticanonically balanced metrics. Combined with the result by Rubinstein--Tian--Zhang, we obtain the following algebro-geometric corollary: the $\delta_m$-invariant of Fujita--Odaka satisfies $\delta_m >1$ if and only if the Fano manifold is stable in the sense of Saito--Takahashi, establishing a Hilbert--Mumford type criterion for $\delta_m >1$. We also extend this result to the K\"ahler--Ricci $g$-solitons and the coupled K\"ahler--Einstein metrics, and as a by-product we obtain a formula for the asymptotic slope of the coupled Ding functional in terms of multiple test configurations.
\end{abstract}

\maketitle

\tableofcontents

\section{Introduction}

\subsection{Background}

Canonical metrics on Fano manifolds have been a focus of intensive studies in recent years, particularly in relation to the stability notions in Geometric Invariant Theory (GIT) \cite{MFK}. There are many important results, but we mention here only the ones that are directly related to the results in this paper. The existence of K\"ahler--Einstein metrics on a Fano manifold with no nontrivial holomorphic vector fields is equivalent to the stability condition called the uniform Ding stability \cite[Theorem A]{BBJ}. This stability condition is given in terms of objects called the test configurations, which is a certain class of degenerations of Fano manifolds, and plays the role of the $\cx^*$-action in the Hilbert--Mumford criterion of stability \cite{MFK}. The uniform Ding stability can be characterised by an invariant called the $\delta$-invariant introduced by Fujita--Odaka \cite{fo18}, in the sense that it is equivalent to $\delta >1$ by \cite{bj20,Fujita2019,fo18}. In a recent paper, K.~Zhang \cite{Zhang21} directly proved that $\delta >1$ implies the existence of K\"ahler--Einstein metrics, without using the uniform Ding stability itself and based instead on a finite dimensional approximation (often called quantisation) which is briefly reviewed below; indeed, all the objects mentioned above have finite dimensional approximations in an appropriate sense.

Donaldson \cite{donproj1} proved a foundational theorem that a K\"ahler metric of constant scalar curvature (such as K\"ahler--Einstein metrics) can be approximated by a sequence of Fubini--Study metrics called the balanced metrics, as long as the automorphism group is discrete. Combined with the theorem due to H.~Luo \cite{Luo} and S.~Zhang \cite{zhang96} (see also Phong--Sturm \cite{PS03}), which proves that a balanced metric exists if and only if the manifold is Chow stable, Donaldson's theorem implies that a K\"ahler manifold admitting a constant scalar curvature K\"ahler metric is asymptotically Chow stable \cite[Corollary 4]{donproj1}. The asymptotic Chow stability can be regarded as a finite dimensional approximation of the $K$-stability, which has a well-understood relationship to the Ding stability, in the sense that the Donaldson--Futaki invariant can be written as the limit of a sequence of Chow weights.

For Fano manifolds, we have a variant of balanced metrics that is called the anticanonically balanced metrics, which was introduced by Donaldson \cite[\S 2.2.2]{donnum09}. The analogue of Donaldson's theorem \cite{donproj1} above was established by Berman--Witt Nystr\"om \cite{BWN14}, who proved that the existence of \ke metrics (resp.~K\"ahler--Ricci $g$-solitons) implies that we can find a sequence of anticanonically balanced metrics that converges to the \ke metric (resp.~the K\"ahler--Ricci $g$-soliton), in the sense of currents; see also \cite{tak15} for the case of solitons. The convergence can in fact be improved to the smooth convergence by \cite[Theorem 1.3 with $N=1$]{tak19} and \cite{Ioos20,Ioos21}. From the point of view of stability, Saito--Takahashi \cite{st19} defined a notion of stability that can be regarded as an anticanonical version of the Chow stability, and proved that it is implied by the existence of anticanonically balanced metrics.

The $\delta$-invariant of Fujita--Odaka also has a finite dimensional approximation called the $\delta_m$-invariant, which originally appeared as an intermediate object for the definition of the $\delta$-invariant. Recently, Rubinstein--Tian--Zhang \cite{rtz} proved that the $\delta_m$-invariant characterises the existence of anticanonically balanced metrics. The theorem of K.~Zhang \cite{Zhang21} mentioned above relies on this result.

While many foundational works so far established a deep understanding of the uniform Ding stability, the $\delta$-invariant, and the K\"ahler--Einstein metrics, the relationship between the finite dimensional analogues of these concepts does not seem to be complete in the sense that no GIT stability condition, given in terms of a Hilbert--Mumford type criterion involving test configurations, has been proved to fully characterise $\delta_m > 1$ or the anticanonically balanced metrics. This is the question that is addressed in this paper.

\subsection{Statement of the results}

The main result of this paper is the following, which can be regraded as the anticanonical version of the theorem by H.~Luo \cite{Luo} and S.~Zhang \cite{zhang96}, and establishes the required correspondence between the anticanonically balanced metrics and the GIT stability.

\begin{theorem} \label{mthst}
Let $m \in \mathbb{N}$ be large enough such that $-mK_X$ is very ample. A Fano manifold $(X , -K_X)$ admits an anticanonically balanced metric at level $m$, which is unique up to $\mathrm{Aut}_0(X)$, if and only if it satisfies the following stability condition: for any very ample test configuration $( \CX , \CL )$ for $(X , -K_X)$ of exponent $m$ we have $\mathrm{Ding} ( \CX , \CL ) + \mathrm{Chow}_m ( \CX , \CL ) \ge 0$, with equality if and only if $( \CX , \CL )$ is product.
\end{theorem}

One direction of the above result, i.e.~the existence of anticanonically balanced metrics implying the stated stability condition, was proved by Saito--Takahashi \cite[Theorem 1.2]{st19}. Thus the main point of the above theorem is the stability implying the anticanonically balanced metrics, although the proof that we give in this paper easily establishes both directions.

The main application of the above theorem, combined with Rubinstein--Tian--Zhang \cite[Theorem 2.3]{rtz}, is the following result.

\begin{corollary} \label{mthstc}
	Suppose that a Fano manifold $X$ has no nontrivial holomorphic vector fields. For $m \in \mathbb{N}$ large enough such that $-mK_X$ is very ample, the $\delta_m$-invariant of Fujita--Odaka satisfies $\delta_m >1$ if and only if $\mathrm{Ding} ( \CX , \CL ) + \mathrm{Chow}_m ( \CX , \CL ) >0$ for any nontrivial very ample test configuration $( \CX , \CL )$ for $(X , -K_X)$ of exponent $m$.
\end{corollary}

While the above is a result in algebraic geometry, it seems that no purely algebro-geometric proof is known at the moment of writing this paper; the proof that we give in \S \ref{scpfmtmc} relies on Theorem \ref{mthst} and \cite[Theorem 2.3]{rtz} which concern the anticanonically balanced metrics in differential geometry.  Note also that, in the above corollary, the triviality of test configurations is defined in a way that is not commonly used recently; see Remark \ref{rmtrtc}.

For the definitions of the terminologies that arise in the above results, the reader is referred to \S \ref{scacbalm}, \S \ref{scdstcst}, and \S \ref{aoidm}. 

We also extend these results to the K\"ahler--Ricci $g$-solitons (Theorem \ref{thkrsbm} and Corollary \ref{crstdm}) and the coupled \ke metrics (Theorem \ref{thckebm} and Corollary \ref{crcpdm}). While the main argument of the proof carries over almost word by word for both cases, it turns out that we need to strengthen the stability condition for the coupled \ke case proposed by Hultgren--Witt Nystr\"om \cite{HWN19}. For this purpose, we first define the notion of a test configuration generated by the $\cx^*$-actions of multiple test configurations (Definition \ref{dftcgcp}). We then define a strengthened version of the coupled Ding invariant (Definition \ref{dfcpdginv}), and prove that it naturally arises as the asymptotic slope of the coupled Ding functional (Theorem \ref{ppcdgas}). It seems natural to expect that this result can be applied to give further results for the coupled \ke metrics, as pointed out at the end of \S \ref{sccdinv}, but we shall only consider its implications to the anticanonically balanced metrics in this paper. The case of K\"ahler--Ricci $g$-solitons is also treated similarly, but the analogue of Corollary \ref{mthstc} is weaker than the full characterisation of $\delta^g_m >1$ (see Corollary \ref{crstdm} and \S \ref{aoidmgs}) because of nontrivial holomorphic vector fields, which is analogous to the well-known phenomenon that varieties with nontrivial holomorphic vector fields are never (uniformly) Ding stable.

\begin{remark}
	We point out how Theorem \ref{mthst} relates to some known results for toric Fano manifolds. Yotsutani \cite{Yotsu2017} proved that the asymptotic Chow \textit{semi}stability implies the Ding \textit{poly}stability for toric Fano manifolds, where we note that there are various results (e.g.~\cite{Ono2011,Yotsu2016}) available for the Chow semistability of toric Fano manifolds. Combined with Theorem \ref{mthst}, we immediately find that asymptotically Chow semistable toric Fano manifolds admit infinitely many anticanonically balanced metrics, by noting that the Chow weight must be zero for product test configurations under such hypotheses (cf.~\cite[Proposition 5.4]{st19} or the proof of Lemma \ref{rmfstgs}); this is not a new result, however, as it also follows from \cite[Theorem 1.7]{BWN14} and \cite[Theorem 1.1]{WZ04} since the higher Futaki invariants \cite{Futaki04} vanish by the Chow semistability \cite[Theorem 5.5]{st19}. 
\end{remark}

It should also be possible to consider the twisted version of above results as in \cite[\S 4.3]{rtz}, but we decide not to treat such cases since Proposition \ref{lmasll}, which is the key estimate in the proof of Theorem \ref{mthst}, seems more complicated with the twist term.

\subsection{Organisation of the paper}

We start by reviewing the differential-geometric preliminaries in \S \ref{scdgprem}. While most of the materials in \S \ref{scdgprem} are well-known, Lemma \ref{lmrfmc} for the coupled Ding functional seems to be a new result that plays an important role later. Algebro-geometric preliminaries are recalled in \S \ref{scagprem}; again this is mostly a review of well-known results, but in \S \ref{sccdinv} we define the notion of a test configuration generated by the $\cx^*$-actions of multiple test configurations, and define the strengthened version of the coupled Ding invariant (Definition \ref{dfcpdginv}) which seems to be new. The relationship between these analytic and algebraic concepts are reviewed in \S \ref{scprcsf}, which is a summary of many well-established foundational results.

The proof of the main results is given in \S \ref{scpfmr}; we note that, almost as a by-product of the proof for the coupled version, we obtain a formula (Theorem \ref{ppcdgas}) for the asymptotic slope of the coupled Ding functional along a coupled psh ray that does not seem to have been considered before and seems appropriate for the study of the coupled K\"ahler--Einstein metrics. While the invariants $\delta$ and $\delta_m$ are important objects that appear in Corollary \ref{mthstc}, we only review them in Appendix \ref{appdmivfo} since the main body of this paper does not quite depend on these invariants.

\bigskip

\noindent \textbf{Acknowledgements.} The author thanks Kento Fujita, Tomoyuki Hisamoto, Eiji Inoue, Yuji Odaka, and Ryosuke Takahashi for helpful discussions. This work is partially supported by JSPS KAKENHI (Grant-in-Aid for Early-Career Scientists), Grant Number JP19K14524.

\section{Differential-geometric preliminaries} \label{scdgprem}

\subsection{K\"ahler--Einstein metrics and Ding functional} \label{sckemdf}

Let $(X , -K_X)$ be a Fano manifold of complex dimension $n$, polarised by the anticanonical bundle. Throughout this paper, we use the additive notation for the tensor product of line bundles and write $\mathrm{Vol}(X)$ for $\int_X c_1(-K_X)^n $. We work with the very ample line bundle $-mK_X$, by choosing $m \in \mathbb{N}$ to be sufficiently large; we shall further assume that $m$ is sufficiently divisible for the coupled \ke case in \S \ref{sccekbm}. We fix a reference hermitian metric $h_0$ on $-mK_X$, rather than $-K_X$, and write $\omega_0 \in c_1 (-K_X)$ for the associated K\"ahler metric re-scaled by $1/m$. This reference metric is assumed to satisfy various extra hypotheses according to the situation under consideration; we further assume $h_0$ to be a Fubini--Study metric defined in \S \ref{scacbalm} for convenience, and moreover to be $T$-invariant when we consider the K\"ahler--Ricci $g$-solitons in \S \ref{scqkrgs}. The coupled \ke case \S \ref{sccekbm} is slightly more complicated as there will be an auxiliary reference metric (\ref{dfrfchm}), but the one in (\ref{dfrfchm2}) is the relevant one which we denote by $h_0$. In any case, without loss of generality we may assume that $h_0$ satisfies all these required properties and fix the notation once and for all to streamline the notational convention.


We define
\begin{equation*}
	\CH := \{ \phi \in C^{\infty} (X , \rl) \mid \omega_0 + \ai \ddbar \phi / 2 \pi >0 \}
\end{equation*}
for the space of smooth K\"ahler metrics in $c_1 (-K_X)$. Note that a hermitian metric $e^{-m \phi} h_0$ on $-mK_X$, with $\phi \in \CH$, corresponds to the K\"ahler metric $\omega_{\phi} := \omega_0 + \ai \ddbar \phi / 2 \pi $.

An elementary yet important observation is that a hermitian metric on $-K_X$ naturally defines a volume form on $X$; more precisely, a hermitian metric on $-K_X$ is a smooth positive section of $\mathcal{H}om_{C^{\infty}_X} ((-K_X) \otimes \overline{(-K_X)} , \cx )$, which corresponds to the one of $K_X \otimes \overline{K_X}$, i.e.~a volume form. To clarify the notation, we shall write $d \mu_{\phi}$ for the volume form on $X$ defined by the hermitian metric $e^{- m \phi} h_0$ on $-mK_X$ with $\phi \in \CH$; note that we have
\begin{equation} \label{eqcmvf}
	d \mu_{\phi} = e^{- \phi } d \mu_0 .
\end{equation}
It is convenient to re-scale $h_0$ by a nonzero constant if necessary so that $\int_X d \mu_0 = 1$.

Another notational convention that we use is
\begin{equation*}
\fint_X f d \mu_{\phi} :=  \left( \int_X d \mu_{\phi} \right)^{-1} \int_{X} f d \mu_{\phi}
\end{equation*}
for an integrable function $f$ on $X$, noting that this is just a re-scaling so as to have $\fint_X d \mu_{\phi} = 1$.

We recall the following definition.

\begin{definition}
	The \textbf{Ding functional} $\mathscr{D} : \CH \to \rl$ is defined by
	\begin{equation*}
	\mathscr{D} (\phi):= \mathscr{L} (\phi) - \mathscr{E} (\phi),
	\end{equation*}
	where $\mathscr{L} , \mathscr{E} : \CH \to \rl$ are defined respectively as
	\begin{equation*}
	\mathscr{L} (\phi) := - \log \int_X d \mu_{\phi},
\end{equation*}
and
\begin{equation*}
	\mathscr{E}(\phi) := \frac{1}{(n+1) \mathrm{Vol} (X)} \sum_{j=0}^n \int_X \phi \omega_0^{n-j} \wedge \omega^j_{\phi} .
\end{equation*}
\end{definition}

A straightforward computation reveals that a critical point of $\mathscr{D}$ is precisely the metric that satisfies $\omega^n_{\phi} =  d \mu_{\phi}$, which is equivalent to the \textbf{K\"ahler--Einstein} equation 
\begin{equation*}
	\mathrm{Ric} (\omega_{\phi}) = \omega_{\phi}.
\end{equation*}
An important result due to Berndtsson \cite{Ber09a,ber15} is that $\mathscr{L}$ is convex along psh (plurisubharmonic) rays in $\CH$, which we recall later (Theorem \ref{thberndt}) in the form that is necessary for the proof of our main results. Note that Berndtsson's convexity implies that the critical point, i.e.~the K\"ahler--Einstein metric, is unique \cite[Theorems 1.2 and 5.1]{ber15}, giving an alternative proof of the result originally obtained by Bando--Mabuchi \cite{BanMab} that can be generalised to more singular situations.

Note finally that the convention $\int_X d \mu_0 = 1$ implies $\mathscr{L}( 0) = \mathscr{E} (0) = \mathscr{D} (0) = 0$.

\subsection{Anticanonically balanced metrics} \label{scacbalm}

An important subset of $\CH$ is the Fubini--Study metrics given in terms of the Kodaira embedding
\begin{equation*}
	\iota : X \inj \prj (H^0 (X , -mK_X)^{\vee} )
\end{equation*}
defined as follows. Suppose that we have a positive definite hermitian form $H$ on $H^0 (X , -mK_X)$. This naturally induces a hermitian metric on the hyperplane bundle $\CO(1)$ over $\prj (H^0 (X , -mK_X)^{\vee} )$, and hence on $-mK_X = \iota^* \CO (1)$ by pullback. By taking the $m$-th root, we get a hermitian metric on $-K_X$ and write $\mathrm{FS} (H)$ for the corresponding K\"ahler potential in $\CH$. Noting that we may write $H = e^{A^*} e^A$ for some $A \in \mathfrak{gl} (H^0 (X , -mK_X))$, the set $\CB_m$ of positive definite hermitian forms on $H^0 (X , -mK_X)$ can be identified with the right coset space
\begin{equation} \label{eqcmbrcs}
	\CB_m = U (N_m) \backslash GL(N_m , \cx ) ,
\end{equation}
where $N_m := \dim_{\cx} H^0 (X , -mK_X)$. Thus, the construction as above defines the \textbf{Fubini--Study map} \cite[\S 2]{donproj2}
\begin{equation*}
	\mathrm{FS} : \CB_m \to \CH,
\end{equation*}
which is injective \cite[Theorem 1.1]{Lempert2021}. We write $\CH_m := \mathrm{FS} (\CB_m)$ for its image, and the elements of $\CH_m$ are called the \textbf{Fubini--Study metrics}. We also note that we have the Hilbert map \cite[\S 2.2]{donnum09},
\begin{equation*}
	\mathrm{Hilb}_{\mu} : \CH \to \CB_m ,
\end{equation*}
defined by associating $\phi \in \CH$ to the (re-scaled) $L^2$-inner product $\frac{N_m}{\int_X d \mu_{\phi}} \int_X e^{-m \phi} h_0 ( \cdot , \cdot ) d \mu_{\phi}$ on $H^0 (X , -mK_X)$.

We can write down the Fubini--Study metric more explicitly as follows. We choose a reference hermitian form $H_0 \in \CB_m$ once and for all, and we define the reference metric $h_0$ on $-mK_X$ to be the Fubini--Study metric defined by $H_0$. More precisely, writing $\tilde{h}_0$ for the hermitian metric on the hyperplane bundle over $\prj (H^0 (X , -mK_X)^{\vee} )$ defined by $H_0$, we write
\begin{equation} \label{dfhtilde}
	h_0 := \iota^* \tilde{h}_0.
\end{equation}
With this reference metric chosen, the Fubini--Study metric $\mathrm{FS} (H) \in \CH_m$ defined by $H \in \CB_m$ can be written as
\begin{equation} \label{eqfs}
	\mathrm{FS} (H) = \frac{1}{m} \log \sum_{i=1}^{N_m} \left\Vert s^H_i \right\Vert^2_{h_0} \in \CH ,
\end{equation}
where $\{ s^H_i \}_{i=1}^{N_m}$ is any $H$-orthonormal basis for $H^0 (X , -mK_X)$. Note that $\mathrm{FS} (H)$ can be characterised as the unique element in $\CH$ such that $h_H := \exp( -m \mathrm{FS} (H))h_0$ satisfies $\sum_{i=1}^{N_m} \left\Vert s^H_i \right\Vert^2_{h_H} = 1$ over $X$.

\begin{remark} \label{rmrfbs}
	In what follows, $\{ s_i \}_{i=1}^{N_m}$ stands for a fixed $H_0$-orthonormal basis for $H^0 (X , -mK_X)$, which we shall refer to as the \textbf{reference basis}, and $H_0$ will often be regarded as the identity matrix with respect to it. For each $A \in \mathfrak{gl} (H^0 (X , -mK_X))$, we write $A_{ij}$ for the $(i,j)$-th entry of the matrix representing $A$ with respect to the reference basis, unless otherwise stated. The hermitian conjugate $A^*$ of $A$ will henceforth be assumed to be with respect to $H_0$; we write $A^{\dagger}$ for the hermitian conjugate with respect to another hermitian form specified in the context. 
\end{remark}

We observe that $\CB_m = U (N_m) \backslash GL(N_m , \cx )$ is a Riemannian symmetric space with respect to the natural bi-invariant metric. An important object is a one-parameter family in $\CH_m$ defined by the geodesic in $\CB_m$ with respect to the bi-invariant metric, defined more explicitly as follows.

\begin{definition}
	 The \textbf{Bergman geodesic ray} emanating from $H_0$ is a family $\{ H_t \}_{t \ge 0} \subset \CB_m$ of positive definite hermitian forms defined by
\begin{equation} \label{eqbggda}
	H_t := e^{-t A^*}e^{-tA},
\end{equation}
for an $H_0$-hermitian endomorphism $A \in \mathfrak{gl} (H^0 (X , -mK_X))$.
\end{definition}
By abuse of terminology, the resulting family $\{ \mathrm{FS} (H_t) \}_{t \ge 0} \subset \CH_m$ emanating from $h_0$ is also called the Bergman geodesic ray. The formula (\ref{eqfs}) immediately yields
\begin{equation} \label{eqfst}
	\mathrm{FS} (H_t) = \frac{1}{m} \log \sum_{i=1}^{N_m} \left\Vert e^{tA} s_i \right\Vert^2_{h_0},
\end{equation}
where $\{ s_i \}_{i=1}^{N_m}$ is any $H_0$-orthonormal basis for $H^0 (X , -mK_X)$, which we may of course take to be the reference basis, by noting that $\{ e^{tA} s_i \}_{i=1}^{N_m}$ is an $H_t$-orthonormal basis.

We recall the following functional defined by Berman--Witt Nystr\"om \cite[\S 4.2.2]{BWN14}, which ``quantises'' the Ding functional.

\begin{definition}
	The \textbf{quantised Ding functional} $\mathscr{D}_m : \CB_m \to \rl$ is defined as
	\begin{equation*}
		\mathscr{D}_m ( H ) := \mathscr{L} (\mathrm{FS} (H)) - \mathscr{E}_m (H)
	\end{equation*}
	where
	\begin{equation*}
		\mathscr{E}_m (H) := - \frac{1}{mN_m} \log \det (H H^{-1}_0).
	\end{equation*}
\end{definition}

The computation of the asymptotic slope of $\mathscr{D}_m$ along Bergman geodesic rays will be of crucial importance in what follows, and hence we record an elementary result for the derivative of $\mathscr{D}_m$.

\begin{lemma} \emph{(cf.~\cite[(7.3) and (7.5)]{BBGZ})} \label{lmdvffs}
Suppose that $\{ H_t \}_{t \ge 0} \subset \CB_m$ is a Bergman geodesic ray (\ref{eqbggda}). Then we have
	\begin{equation*}
		\frac{d}{dt} \mathscr{L} (\mathrm{FS} (H_t)) = \frac{1}{m}  \sum_{i,j=1}^{N_m} (A_{ij} + A^*_{ij}) \fint_X   \frac{ h_0(e^{tA} s_i , e^{tA}s_j )}{\sum_{l=1}^{N_m} \Vert e^{tA} s_l \Vert^2_{h_0}} d \mu_{\mathrm{FS} (H_t)}, 
	\end{equation*}
	and
	\begin{equation*}
		\frac{d}{dt} \mathscr{E}_m (H_t) = \frac{\mathrm{tr} (A^* +A)}{mN_m}.
	\end{equation*}
	\end{lemma}

\begin{proof}
The first formula is obvious from
\begin{equation*}
		\frac{d}{dt} \mathscr{L} (\mathrm{FS} (H_t)) = \left( \int_X d \mu_{\mathrm{FS} (H_t)} \right)^{-1} \int_X \left( \frac{d}{dt} \mathrm{FS} (H_t) \right) d \mu_{\mathrm{FS} (H_t)}
\end{equation*}
and the equation (\ref{eqfst}). The second is also immediate from the definition.
\end{proof}

Note that the above lemma immediately implies
\begin{equation} \label{eqemqaff}
		\frac{d^2}{dt^2} \mathscr{E}_m (H_t) = 0.
\end{equation}

Recall that the \textbf{Bergman function} $\rho_m (\mathrm{FS}(H))
 \in C^{\infty} (X , \rl)$ (also called the distortion function in \cite[\S 7.3]{BBGZ}) is defined as follows: writing $\{ \sigma_i \}_{i=1}^{N_m}$ for an orthonormal basis for $H^0 (X , -mK_X)$ with respect to $\frac{\mathrm{Vol}(X)}{N_m}\mathrm{Hilb}_{\mu} \circ \mathrm{FS} (H)$, we define
\begin{equation} \label{eqdberg}
	\rho_m (\mathrm{FS}(H)):= \sum_{i=1}^{N_m} \Vert \sigma_i \Vert^2_{h_H},
\end{equation}
where $h_H := \exp( -m \mathrm{FS} (H)) h_0$. The following well-known result is contained in \cite[\S 7.3]{BBGZ}, and also stated slightly implicitly in  \cite{st19,tak19}.

\begin{proposition} \label{ppacbal}
The following are equivalent for $H \in \CB_m$.
\begin{enumerate}
\renewcommand{\labelenumi}{(\roman{enumi})}
		\item There exists an $H_0$-hermitian $A \in \mathfrak{gl} (H^0 (X , -mK_X))$ such that $H = e^{-A^*} e^{-A}$ and that for all $i,j=1, \dots , N_m$ we have
		\begin{equation*} 
			 \fint_{X} \frac{h_0 (e^A s_i , e^A s_j) }{\sum_{l=1}^{N_m} \Vert e^A s_l \Vert^2_{h_0}} d \mu_{\mathrm{FS}(H)} - \frac{1}{N_m} \delta_{ij} = 0,
		\end{equation*}
		where $\delta_{ij}$ is the Kronecker delta.
		\item The Bergman function $\rho_m (\mathrm{FS}(H))$ is constant over $X$.
		\item $\mathrm{Hilb}_{\mu} \circ \mathrm{FS} (H) = H$.
		\item $H \in \CB_m$ is a critical point of $\mathscr{D}_m$.
	\end{enumerate}
\end{proposition}

\begin{proof}
	The proof of $\text{(i)} \Leftrightarrow \text{(iv)}$ is an obvious consequence of Lemma \ref{lmdvffs}, which gives
\begin{equation*}
	\frac{d}{dt} \mathscr{D}_m (H_t) = \frac{1}{m} \sum_{i,j=1}^{N_m} (A_{ij} + A^*_{ij}) \left(  \fint_{X} \frac{h_0 (e^{tA} s_i , e^{tA} s_j) }{\sum_{l=1}^{N_m} \Vert e^{tA} s_l \Vert^2_{h_0}} d \mu_{\mathrm{FS}(H_t)} - \frac{1}{N_m} \delta_{ij} \right).
\end{equation*}
$\text{(i)} \Leftrightarrow \text{(iii)}$ immediately follows from (\ref{eqfs}) and the definition of $\mathrm{Hilb}_{\mu}$, and $\text{(ii)} \Leftrightarrow \text{(iii)}$ is obvious from (\ref{eqfs}) and (\ref{eqdberg}).
\end{proof}

\begin{definition}
	A Fubini--Study metric $\mathrm{FS} (H) \in \CH_m$ is said to be \textbf{anticanonically balanced} at level $m$ if it satisfies one of the equivalent conditions in Proposition \ref{ppacbal}.
\end{definition}

The item (i) of Proposition \ref{ppacbal} can be regarded as the ``zero of the moment map'' condition, and $\mathscr{D}_m$ can be regarded as the Kempf--Ness type functional. It is thus natural to expect that the existence of the anticanonically balanced metric can be characterised by a Hilbert--Mumford type criterion in Geometric Invariant Theory. The appropriate criterion, as it turns out, is the $F$-stability defined by Saito--Takahashi \cite{st19} that we review in Definition \ref{dffstst}. For more details on the anticanonically balanced metrics, the reader is referred to \cite[\S 7]{BBGZ}.


\begin{remark} \label{rmdgtrinv}
	Note that the functionals $\mathscr{D}$ and $\mathscr{D}_m$ are translation invariant, in the sense that they satisfy $\mathscr{D} (\phi) = \mathscr{D} (\phi + c)$ and $\mathscr{D}_m (H) = \mathscr{D}_m (e^c H)$ for any $c \in \rl$, which follows immediately from their definitions.
\end{remark}

We finally note that Berndtsson's convexity (recalled later in Theorem \ref{thberndt}) and (\ref{eqemqaff}) immediately imply that a critical point of $\mathscr{D}_m$ must be the global minimum over $\CB_m$.

\subsection{K\"ahler--Ricci $g$-solitons and balanced metrics}  \label{scqkrgs}

We start with the preliminary materials on the automorphism group of Fano manifolds. We write $\mathrm{Aut}_0 (X)$ for the identity component of the automorphism group of $X$. First observe that $\mathrm{Aut}_0 (X) = \mathrm{Aut}_0 (X , -K_X)$ is a linear algebraic group, since the action $\mathrm{Aut}_0 (X) \actson X$ naturally lifts to the anticanonical bundle, which is ample since $X$ is Fano. This in turn implies that we have the action $\mathrm{Aut}_0 (X) \actson \prj (H^0 (X , -mK_X)^{\vee} )$ which preserves the image $\iota(X) \subset \prj (H^0 (X , -mK_X)^{\vee} )$ by the equivariant embedding theorem \cite{Kambayashi} (see also \cite[\S 5.1]{CG}), for any $m \in \mathbb{N}$ such that $-mK_X$ is very ample. More precisely, for $m$ large enough there exists a faithful rational representation
\begin{equation} \label{eqautffebd}
	\theta : \mathrm{Aut}_0 (X) \inj GL (H^0 (X , -mK_X)),
\end{equation}
which is unique up to the choice of the linearisation (i.e.~an overall constant multiple by $\cx^*$), such that it is equivariant with respect to $\iota$, i.e.~for any $x \in X$ and $a \in \mathrm{Aut}_0 (X) $ we have
\begin{equation*} 
	\iota (a \cdot x) = \theta (a) \cdot \iota (x).
\end{equation*}
We observe that an element of $GL (H^0 (X , -mK_X))$ which preserves $\iota(X) \subset \prj (H^0 (X , -mK_X)^{\vee} )$ must be contained in the image of $\theta$ above.

\begin{lemma} \label{lmautaobm}
The group action $\mathrm{Aut}_0 (X) \actson \CB_m$ given by the representation (\ref{eqautffebd}) as
\begin{equation} \label{eqautaobm}
	a \cdot H := \theta(a^{-1})^{\dagger} H \theta(a^{-1}),
\end{equation}
where $\theta(a^{-1})^{\dagger}$ is the hermitian conjugate of $\theta(a^{-1})$ with respect to $H$, is an isometry with respect to the bi-invariant metric on $\CB_m$ which is consistent with the natural action $\mathrm{Aut}_0 (X) \actson \CH$.
\end{lemma}


\begin{proof}
	Recalling the isomorphism $\CB_m = U (N_m) \backslash GL(N_m , \cx )$ in (\ref{eqcmbrcs}) given by $H = e^{B^{\dagger}}  e^{B}$, (\ref{eqautaobm}) is exactly the action on $\CB_m$ given by the right multiplication $e^{B} \mapsto e^{B} \cdot  \theta(a^{-1}) $ and hence is an isometry with respect to the bi-invariant metric. It is consistent with the natural action $\mathrm{Aut}_0 (X) \actson \CH$ since $a^* \mathrm{FS}(H) = \mathrm{FS} (a \cdot H)$ by \cite[Lemma 8]{yhextremal}.
\end{proof}

In what follows, we shall mostly regard $\mathrm{Aut}_0 (X)$ as a closed subgroup of $GL (H^0 (X , -mK_X))$ by (\ref{eqautffebd}), and suppress $\theta$ in the notation. Likewise, its Lie algebra $\mathfrak{aut} (X)$ will be regarded as a Lie subalgebra of $\mathfrak{gl} (H^0 (X , -mK_X))$. We define a maximal compact subgroup $K_0 := \mathrm{Aut}_0 (X) \cap U (N_m)$ of $\mathrm{Aut}_0 (X)$, where the unitarity is defined with respect to the reference hermitian form $H_0 \in \CB_m$. Note that $K_0$ is the isometry group of the reference K\"ahler metric $\omega_0$ by Lemma \ref{lmautaobm} and \cite[Theorem 1.1]{Lempert2021}. We also define a reductive subgroup $\mathrm{Aut}_0 (X)_r := K_0^{\cx}$ of $\mathrm{Aut}_0 (X)$ by its complexification, and note that $\mathrm{Aut}_0 (X)$ can be written as a semidirect product of its unipotent radical and $\mathrm{Aut}_0(X)_r$ by the Jordan--Chevalley decomposition. We write $\mathfrak{aut}(X)_r$ for the Lie algebra of $\mathrm{Aut}_0(X)_r$, and note that if $A \in \mathfrak{aut}(X)$ is $H_0$-hermitian then it must be contained in $\mathfrak{aut}(X)_r$.

For any connected subgroup $K$ of $K_0$, we define
\begin{equation*}
	\CB_m^K := \{ H \in \CB_m \mid \text{$H$ commutes with all elements in $K$.} \},
\end{equation*}
where we identified $H \in \CB_m$ with a hermitian endomorphism with respect to the reference basis. Note that we have $u \cdot H = H$ for all $u \in K$ and $H \in \CB_m^K$, with the action $u \cdot H$ given by (\ref{eqautaobm}), since $u \cdot H = u^{\dagger} H u = u^{-1} H u$ for all $u \in K$ by $K \subset U(N_m)$, noting that $u^{\dagger}$ agrees with the $H_0$-hermitian conjugate $u^*$ since $u$ commutes with $H$. Defining $\CH^K$ for the set of K\"ahler potentials that are invariant under the $K$-action and setting $\CH_m^K := \CH^K \cap \CH_m$, we thus get $\mathrm{FS} (\CB_m^K) \subset \CH^K_m$ by Lemma \ref{lmautaobm}. Just as in (\ref{eqcmbrcs}), we can write $\CB^K_m$ as a right coset space
\begin{equation*}
	\CB_m^K = C(K) \backslash C (K)^{\cx} 
\end{equation*}
where $C(K)$ is the centraliser of $K$ in $U(N_m)$ and $C (K)^{\cx}$ is its complexification. In particular, $\CB^K_m$ is a Riemannian symmetric space with respect to the bi-invariant metric.

We now define the K\"ahler--Ricci $g$-solitons, following \cite[\S 2.2]{BWN14} (see also \cite[\S 6.2]{rtz}). Let $T$ be a (compact) real torus in $\mathrm{Aut}_0 (X)$, and suppose that $(X,-K_X)$ admits a Hamiltonian (with respect to $\omega_0$) $T$-action which is also holomorphic; note that this necessarily implies that $T$ is a subgroup of $K_0$. We write $T^{\cx}$ for the complexification of $T$, $\mathrm{Aut}_0 (X,T^{\cx})$ for the elements in $\mathrm{Aut}_0 (X)$ that commute with $T^{\cx}$, and $\mathfrak{aut} (X,T^{\cx})$ for its Lie algebra. Just as in Lemma \ref{lmautaobm}, we have a natural action of $\mathrm{Aut}_0 (X,T^{\cx} )$ on $\CB_m^T$ which is an isometry with respect to the bi-invariant metric and consistent with the natural action on $\CH^T$. As before, we define its reductive subgroup $\mathrm{Aut}_0 (X,T^{\cx} )_r$ by the complexification of  $\mathrm{Aut}_0 (X,T^{\cx} ) \cap U(N_m)$, whose Lie algebra is denoted by $\mathfrak{aut} (X , T^{\cx})_r$.

The action $T \actson (X,-K_X)$, with the K\"ahler form $\omega_{\phi}$ ($\phi \in \CH^T$), defines a \textbf{moment map}
\begin{equation*}
m_{\phi} : X \to \Ft^{\vee} \cong \rl^{\dim T} ,
\end{equation*}
and its image $P:= m_{\phi} (X)$ is a compact convex polytope in $\rl^{\dim T}$ known as the Delzant polytope. The \textbf{Duistermaat--Heckman measure} $\nu$ is the measure on $\rl^{\dim T}$, supported on $P$, defined by
\begin{equation*}
\nu := (m_{\phi} )_* \left( \frac{1}{\mathrm{Vol} (X)} \frac{\omega^n_{\phi}}{n!} \right)
\end{equation*}
which is known to be absolutely continuous and independent of $\phi$ \cite{DuiHec82}.

\begin{definition}
For a $T$-invariant K\"ahler potential $\phi \in \CH^T$ and a smooth function $g : P \to \rl_{>0}$, the \textbf{$g$-Monge--Amp\`ere measure} $\mathrm{MA}_g (\phi)$ is a smooth volume form on $X$ defined by
\begin{equation*}
\mathrm{MA}_g (\phi) := \frac{g(m_{\phi})}{\mathrm{Vol} (X)} \frac{\omega^n_{\phi}}{n!} .
\end{equation*}
\end{definition}

We can define $\mathrm{MA}_g (\phi)$ for a more general singular potential $\phi$ as in \cite[\S 2.2]{BWN14} but we only consider the smooth case. Note that the definition of $\nu$ means that we have the volume normalisation
\begin{equation*}
\int_X \mathrm{MA}_g (\phi) = \int_P g \nu .
\end{equation*}

\begin{definition}
Let $m_{\phi} : X \to \Ft^{\vee}$ be a moment map for the torus action with respect to the \kah form $\omega_{\phi}$, and $g : P \to \rl_{>0}$ be a smooth function. We say that $\phi \in \CH^T$ is a \textbf{K\"ahler--Ricci $g$-soliton} if it satisfies
\begin{equation*}
\mathrm{Ric} (\omega_{\phi}) = \omega_{\phi} + \ai \ddbar \log g (m_{\phi}) .
\end{equation*}
\end{definition}

An equivalent way of writing down the above equation is
\begin{equation*}
 \frac{\mathrm{MA}_g (\phi)}{\int_P g \nu} = e^{- \phi + f_{\phi} } \frac{d \mu_{\phi}}{\int_X d \mu_{\phi}} ,
\end{equation*}
where $f_{ \phi}$ is the Ricci potential, which is the unique ($T$-invariant) smooth function on $X$ which satisfies
\begin{equation*}
\mathrm{Ric} (\omega_{\phi}) = \omega_{\phi} + \ai \ddbar f_{ \phi}  \quad \text{and} \quad \fint_X e^{f_{ \phi}} d \mu_{\phi} = 1.
\end{equation*}
The above description in terms of the $g$-Monge--Amp\`ere measure can be easily extended to allow for more singular solutions, as explained in \cite[\S 2.2]{BWN14}.

\begin{remark}
The K\"ahler--Ricci $g$-soliton reduces to the usual K\"ahler--Ricci soliton for an appropriate choice of $g$ (see \cite[\S 2.3]{BWN14}, \cite[(2.35)]{Hisamotomab}) when we take $m_{\phi}$ to be the Hamiltonian function for the real part of the soliton vector field (i.e.~the holomorphy potential of the soliton vector field); recall that the soliton vector field is a holomorphic vector field uniquely determined by the volume minimisation, as proved by Tian--Zhu \cite[Lemma 2.2]{TianZhu2002}. We may assume that $T^{\cx}$ contains this soliton vector field, as it provides the most interesting examples, although it is not necessary for the proof of the results in this paper. Note also that the holomorphy potential is a real function as long as the K\"ahler metrics under consideration are $T$-invariant. The K\"ahler--Ricci $g$-solitons also include the Mabuchi soliton, which was defined by Mabuchi \cite{Mab2001} and seems to attract renewed attention after the work of Y.~Yao \cite{Yaomab}, by taking $g$ to be as in \cite[(2.32)]{Hisamotomab}.
\end{remark}


\begin{remark}
	Han--Li \cite{HanLi20} (and Hisamoto \cite{Hisamotomab} for the case of Mabuchi solitons) proved that the existence of K\"ahler--Ricci $g$-solitons is equivalent to a version of uniform stability.
\end{remark}

Following \cite[\S 2.6]{BWN14}, we define a functional $\mathscr{D}^{g} : \CH^T \to \rl$ by
\begin{equation*}
	\mathscr{D}^{g} (\phi) :=   \mathscr{L} (\phi) - \mathscr{E}^g (\phi),
\end{equation*}
where $\mathscr{E}^g$ is defined by its first variation as
\begin{equation} \label{eqgmaeg}
\delta \mathscr{E}^g |_{\phi} (\psi) := \int_X \psi \mathrm{MA}_g (\phi) =  \int_X \psi  \frac{g(m_{\phi})}{\mathrm{Vol} (X)}  \frac{\omega^n_{\phi}}{n!}.
\end{equation}
That $\mathscr{E}^g$ is well-defined is proved by Berman--Witt Nystr\"om \cite[Lemma 2.14]{BWN14}, where the case of the K\"ahler--Ricci solitons was originally proved by X.~Zhu \cite[Lemma 3.1]{Zhu2000}. It then follows that $\phi \in \CH^T$ is a critical point of $\mathscr{D}^g$ if and only if it is a K\"ahler--Ricci $g$-soliton. See \cite[\S 2.4]{BWN14} for more details.

The above functional can be ``quantised'', as proposed by \cite[\S 4]{BWN14}. First we write
\begin{equation*}
P_m := \{ \lambda \in \mathrm{Hom} (T^{\cx}, \cx^*) \mid \lambda \text{ appears as a weight of } T^{\cx} \actson H^0 (X , -mK_X) \}.
\end{equation*}
It is well-known that $\frac{1}{m} P_m \subset P = m_{\phi} (X)$, as in \cite[Lemma 13]{Lah19}. We have the weight decomposition
\begin{equation} \label{eqwtdct}
H^0 (X , -mK_X ) = \bigoplus_{\lambda \in P_m} R_{m ,\lambda}
\end{equation}
where $T^{\cx}$ acts with weight $\lambda \in P_m$ on $R_{m , \lambda}$. We also write
\begin{equation*}
N_{m , \lambda} := \dim_{\cx} R_{m , \lambda}
\end{equation*}
and
\begin{equation*}
\overline{g_m} := \frac{1}{N_m} \sum_{\lambda \in P_m} g ( \lambda / m) N_{m , \lambda}.
\end{equation*}

\begin{definition}
The \textbf{quantised Ding functional for the K\"ahler--Ricci $g$-soliton} is a map $\mathscr{D}^{g}_m : \CB_m^T \to \rl$ defined by 
\begin{equation*}
	\mathscr{D}^{g}_m (H) := \mathscr{L} (\mathrm{FS}(H)) - \mathscr{E}^g_m (H),
\end{equation*}
where $\mathscr{E}_m^g : \CB_m^T \to \rl$ is defined by
\begin{equation*}
\mathscr{E}^g_m (H) := -  \frac{1}{m N_m \overline{g_m}} \sum_{\lambda \in P_m} g ( \lambda / m) \log \det \left(HH_0^{-1} |_{R_{m , \lambda}} \right) .
\end{equation*}
An \textbf{anticanonically $g$-balanced metric} at level $m$ is a critical point of $\mathscr{D}^{g}_m$.
\end{definition}

Note that, for the Bergman geodesic ray $\{ H_t \}_{t \ge 0} \subset \CB^T_m$ as defined in (\ref{eqbggda}), we have
\begin{equation} \label{eqslegm}
	\frac{d}{dt} \mathscr{E}^g_m (H_t) =  \frac{1}{m N_m \overline{g_m}} \sum_{\lambda \in P_m} g ( \lambda / m) \mathrm{tr} \left( (A+A^*)|_{R_{m , \lambda}} \right) .
\end{equation}

Just as we saw in Proposition \ref{ppacbal}, we can characterise the critical point of $\mathscr{D}^g_m$ by the ``zero of the moment map'' condition. Suppose that we write $\left\{ s_{\alpha}^{(\lambda )} \right\}_{\alpha =1}^{N_{m , \lambda}}$ for an $H_0$-orthonormal basis for each $R_{m ,\lambda}$, noting that the weight subspaces are orthogonal to each other with respect to $H_0$ which is $T$-invariant. Then $H = e^{-A^*} e^{-A} \in \CB^T_m$ is a critical point of $\mathscr{D}^g_m$ if and only if
\begin{equation*} 
 \fint_{X} \frac{h_0 \left( e^A s^{(\lambda )}_{\alpha} , e^A s^{(\lambda )}_{\beta} \right) }{\sum_{\lambda' \in P_m} \sum_{\gamma = 1}^{N_{m , \lambda'}} \left\Vert e^A s^{(\lambda' )}_{\gamma} \right\Vert^2_{h_0}} d \mu_{\mathrm{FS}(H)} - \frac{g(\lambda / m)}{m N_m \overline{g_m}} \delta_{\alpha \beta} = 0,
\end{equation*}
for all $\alpha , \beta = 1 , \dots , N_{m , \lambda}$ and all $\lambda \in P_m$; note that $A$ preserves the weight decomposition (\ref{eqwtdct}) as $H \in \CB^T_m$. We can also characterise the above as the metric for which the $g$-Bergman function \cite[\S 4.2.1]{BWN14} is constant, by arguing exactly as in Proposition \ref{ppacbal}.

Just as we pointed out for $\mathscr{D}_m$ in \S \ref{scacbalm}, a critical point of $\mathscr{D}^{g}_m$ is necessarily the global minimum over $\CB^T_m$; in fact it attains the global minimum over $\CB_m$, as pointed out in \cite[Proof of Theorem 2.11]{rtz}.

\subsection{Coupled K\"ahler--Einstein metrics and balanced metrics} \label{sccekbm}

Suppose that we have a $k$-tuple of ample $\rtn$-line bundles $L_1 , \dots , L_k$ over a Fano manifold $X$ such that
\begin{equation*}
	- K_X =  L_1 + \cdots +  L_k .
\end{equation*}
We also write $(X , -K_X; L_1 , \dots , L_k)$ for the above data.

Coupled \ke metrics were introduced by Hultgren--Witt Nystr\"om \cite{HWN19} as a generalisation of the \ke metrics which is compatible with the above ``decomposition'' of $-K_X$, and have been actively studied recently. We recall here some basic results established in \cite{HWN19}, and also its relationship to the geometric quantisation as given by Takahashi \cite{tak19}; in fact we will complement it by adding some new materials that were not discussed in \cite{tak19}, which turns out to be important later.

We choose a positive integer $m$ to be sufficiently large and divisible so that each of $-mK_X$, $m L_1 , \dots , m L_k$ is a very ample line bundle and that the natural multiplication map
\begin{equation*}
	H^0 (X , m L_1 ) \otimes \cdots \otimes  H^0 (X , m L_k) \to H^0(X , m L_1 + \cdots + m L_k ) = H^0(X , -mK_X)
\end{equation*}
is surjective (see e.g.~\cite[Example 1.2.22]{LazarsfeldI}). An elementary observation, which follows from taking the dual of the above, is that we have a sequence of embeddings
\begin{equation} \label{mlmp}
	X \inj \prj(H^0 (X , -mK_X)^{\vee}) \inj \prj (H^0 (X , m L_1 )^{\vee} \otimes \cdots \otimes  H^0 (X , m L_k)^{\vee})
\end{equation}
which turns out to be important. We write $\iota_{\mathrm{coupled}}$ for the composition of the two embeddings above. Note that the hyperplane bundle over $\prj (H^0 (X , m L_1 )^{\vee} \otimes \cdots \otimes  H^0 (X , m L_k)^{\vee})$ is pulled back by $\iota_{\mathrm{coupled}}$ to $-mK_X$ since the second embedding is linear.

On the other hand, since $m L_1 , \dots , m L_k$ are very ample, for each $i=1, \dots , k$ we have the Kodaira embedding $\iota_i : X \inj \prj ( H^0 (X , m L_i)^{\vee})$. For each $i$ we fix a reference hermitian form $H_{i,0}$ for $H^0(X , mL_i )$, and define $\tilde{h}_{i,0}$ to be the Fubini--Study metric on $\prj ( H^0 (X , m L_i)^{\vee})$ with respect to $H_{i,0}$, which is pulled back to a hermitian metric $h_{i,0}:=\iota_i^* \tilde{h}_{i,0}$ on $mL_i$ by $\iota_i$, entirely analogously to (\ref{dfhtilde}) in \S \ref{scacbalm}. We write $\theta_i \in c_1 (L_i)$ for the associated K\"ahler metric, scaled by $1/m$ to be in $c_1 (L_i)$. As in (\ref{eqfs}), for a general positive definite hermitian form $H_i$ on $H^0 (X , mL_i)$ we define
\begin{equation} \label{eqfscpdi}
	\mathrm{FS}_i (H_i) := \frac{1}{m} \log \sum_{j=1}^{N_{i,m}} \left\Vert s^{H_i}_j \right\Vert^2_{h_{i,0}},
\end{equation}
where $\{ s^{H_i}_j \}_{j=1}^{N_{i,m}}$ is any $H_i$-orthonormal basis for $H^0 (X , mL_i)$ and we wrote
\begin{equation*}
	N_{i,m} := \dim_{\cx} H^0 (X , m L_i).
\end{equation*}

We pick an auxiliary reference hermitian metric on $-mK_X$ to be
\begin{equation} \label{dfrfchm}
	h'_0 := h_{1,0} \otimes \cdots \otimes h_{k,0},
\end{equation}
with the associated volume form $d \mu'_0$ which we may assume has unit volume over $X$ by re-scaling; a slightly subtle point is that, while this is a metric on $-mK_X$ that is naturally determined by the reference hermitian metrics on $mL_1 , \dots , mL_k$, we need another reference metric $h_0$ on $-K_X$, defined later in (\ref{dfrfchm2}).

With the reference metrics chosen as above, we define the space of ``coupled'' K\"ahler potentials defined as
\begin{equation*}
\bm{\mathcal{H}} := \CH_1 \times \cdots \times \CH_k,
\end{equation*}
where for each $i=1 , \dots , k$ we define
\begin{equation*}
	\CH_i := \{ \phi \in C^{\infty} (X , \rl) \mid \theta_i + \ai \ddbar \phi / 2 \pi >0 \}.
\end{equation*}

\begin{definition}
Suppose $- K_X =  L_1 + \cdots +  L_k$. A $k$-tuple of K\"ahler metrics $\omega_1 , \dots , \omega_k$, $\omega_i \in  c_1 ( L_i)$ for $i=1 , \dots , k$, is said to be \textbf{coupled K\"ahler--Einstein} if
\begin{equation*}
\mathrm{Ric} (\omega_1) = \cdots = \mathrm{Ric} (\omega_k) = \sum_{i=1}^k \omega_i .
\end{equation*}
\end{definition}

Hultgren--Witt Nystr\"om \cite[Proposition 2.8]{HWN19} proved (see also Remark \ref{rmrfmchwn} below) that the above metric is precisely the critical point of the following functional.

\begin{definition}
The \textbf{coupled Ding functional} is a map $\mathscr{D}^{\mathrm{coupled}} : \bm{\CH} \to \rl$ defined by
\begin{equation*}
\mathscr{D}^{\mathrm{coupled}} (\phi_1 , \dots , \phi_k) := \mathscr{L}^{\mathrm{coupled}} (\phi_1 , \dots , \phi_k) - \sum_{i=1}^k \mathscr{E} (\phi_i),
\end{equation*}
where
\begin{equation*}
\mathscr{L}^{\mathrm{coupled}} (\phi_1 , \dots , \phi_k ) := -\log \int_X \exp \left( - \sum_{i=1}^k \phi_i \right) d \mu'_0.
\end{equation*}
\end{definition}
 
Note that Pingali \cite{Pin18} reduced the existence of coupled \ke metrics to a priori $C^0$-estimates.

\begin{remark} \label{rmrfmchwn}
	The choice of the reference metric is in fact a slightly subtle point, as it is used to show that the Euler--Lagrange equation for $\mathscr{D}^{\mathrm{coupled}}$ is exactly the coupled \ke equation \cite[Lemma 2.1 and Proposition 2.8]{HWN19}. In \cite[\S 2.1]{HWN19} the reference hermitian metric on $-K_X$ is chosen to be the volume form $\eta^n$ (normalised to have unit volume over $X$) given by a K\"ahler metric $\eta$ satisfying
	\begin{equation*}
	\mathrm{Ric}(\eta) = \sum_{i=1}^k  \theta_i,
	\end{equation*}
	which exists by Yau's theorem \cite{Yau78}. This is in fact the same as our choice of the reference (\ref{dfrfchm}) since $\eta^n = d \mu'_0$; note that the K\"ahler metric associated to $h'_0$, re-scaled to be in $c_1 (-K_X)$, satisfies
	\begin{equation*}
		- \frac{1}{m} \ai \ddbar \log h'_0 = - \frac{1}{m}  \ai \ddbar \log h_{1,0} \otimes \cdots \otimes h_{k,0} = \sum_{i=1}^k \theta_i,
	\end{equation*}
	which immediately implies $\eta^n = d \mu'_0$ since $\mathrm{Ric}(\eta) = \sum_{i=1}^k  \theta_i$ and both volume forms have unit volume over $X$.
\end{remark}


Just as we did in \S \ref{scacbalm}, for each $i = 1 , \dots , k$ we define the right coset space
\begin{equation*}
\CB_{i,m} :=  U(N_{i,m}) \backslash GL(N_{i,m} , \cx ),
\end{equation*}
which is the set of positive definite hermitian forms on $H^0 (X , m L_i)$. We further set
\begin{equation*}
\bm{\CB}_{m} := \CB_{1, m} \times \cdots \times \CB_{k,m} ,
\end{equation*}
which is a Riemannian symmetric space with the bi-invariant metric, and define $\bm{\CH}_m \subset \bm{\CH}$ by
\begin{equation*}
	\bm{\CH}_m := \mathrm{FS}_1 (\CB_{1,m}) \times \cdots \times \mathrm{FS}_k (\CB_{k,m}) .
\end{equation*}

So far we mostly considered $k$-tuples of Fubini--Study metrics associated to the very ample line bundles $mL_1 , \dots , m L_k$. On the other hand, we have the embedding
\begin{equation*}
	\iota_{\mathrm{coupled}} : X  \inj \prj (H^0 (X , m L_1 )^{\vee} \otimes \cdots \otimes  H^0 (X , m L_k)^{\vee})
\end{equation*}
given by (\ref{mlmp}). Recalling that we fixed a reference hermitian form $H_{i,0} \in \CB_{i,m}$ for each $i=1 , \dots , k$, we can define a reference hermitian form
\begin{equation*}
	\bm{H}_0 := H_{1,0} \otimes \cdots \otimes H_{k,0}
\end{equation*}
on $H^0 (X , m L_1 ) \otimes \cdots \otimes  H^0 (X , m L_k)$. Just as in Remark \ref{rmrfbs}, we fix a reference basis for $H^0 (X , m L_1 ) \otimes \cdots \otimes  H^0 (X , m L_k)$ as an $\bm{H}_0$-orthonormal basis which we write as $\{ \bm{s}_{\bm{j}} \}_{\bm{j}}$; the multi-index $\bm{j}$ here is such that
\begin{equation*}
	\bm{j}:=(j_1 , \dots , j_k) , \quad j_i \in \{ 1 , \dots , N_{i,m}\} \text{ for each } i=1,  \dots , k ,
\end{equation*}
and
\begin{equation} \label{eqcprbsb}
	\bm{s}_{\bm{j}} := s_{j_1} \otimes \cdots \otimes s_{j_k} ,
\end{equation}
where $\{ s_{j_i} \}_{j_i=1}^{N_{i,m}}$ is an $H_{i,0}$-orthonormal basis for $H^0 (X , mL_i)$. Recalling also that the hyperplane bundle is pulled back to $-mK_X$ by $\iota_{\mathrm{coupled}}$, we see that $\bm{H}_0$ defines a hermitian metric $\tilde{h}_0$ on the hyperplane bundle over $\prj (H^0 (X , m L_1 )^{\vee} \otimes \cdots \otimes  H^0 (X , m L_k)^{\vee})$ which defines a hermitian metric
\begin{equation} \label{dfrfchm2}
	h_0 := \iota^*_{\mathrm{coupled}} \tilde{h}_0
\end{equation}
on $-mK_X$, with the associated volume form $d \mu_0$. We take this $h_0$ to be the reference hermitian metric on $-mK_X$. Writing $\omega_0 \in c_1 (-K_X)$ for the associated K\"ahler metric (scaled by $1/m$), we take $\omega_0$ to be the reference K\"ahler metric for the K\"ahler class $c_1 (-K_X)$, and write
\begin{equation*}
	\CH = \{ \phi \in C^{\infty} (X , \rl) \mid \omega_0 + \ai \ddbar \phi / 2 \pi >0 \} ,
\end{equation*}
continuing with the notation so far.

For a general element $(H_1 , \dots , H_k) \in \bm{\CB_m}$, we define $\bm{H} := H_1 \otimes \cdots \otimes H_k$ for the positive definite hermitian form on $H^0 (X , m L_1 ) \otimes \cdots \otimes  H^0 (X , m L_k)$. We also write $\left\{ \bm{s}^{\bm{H}}_{\bm{j}} \right\}_{\bm{j}}$ for the associated orthonormal basis, where
\begin{equation*}
	\bm{s}^{\bm{H}}_{\bm{j}} := s^{H_1}_{j_1} \otimes \cdots \otimes s^{H_k}_{j_k}
\end{equation*}
and $\left\{ s^{H_i}_{j_i} \right\}_{j_i=1}^{N_{i,m}}$ is an $H_{i}$-orthonormal basis for $H^0 (X , mL_i)$. The Fubini--Study metric defined by $\bm{H}$ and $\iota_{\mathrm{coupled}}$ is given by 
\begin{equation} \label{eqfscpd}
	\mathrm{FS} (\bm{H}) = \frac{1}{m} \log \left( \sum_{i=1}^k \sum_{j_i=1}^{N_{i,m}} \left\Vert \bm{s}^{\bm{H}}_{\bm{j}} \right\Vert^2_{h_0} \right) \in \CH ,
\end{equation}
just as in the usual case (\ref{eqfs}).

Thus, we find that a $k$-tuple of hermitian forms $(H_1 , \dots , H_k) \in \bm{\CB_m}$ gives rise to a $k$-tuple of K\"ahler potentials $(\mathrm{FS}_1 (H_1) , \dots , \mathrm{FS}_k (H_k)) \in \bm{\CH}_m$ and also an additional K\"ahler potential $\mathrm{FS} (\bm{H}) \in \CH$. An intriguing relationship satisfied by these metrics is the following lemma, which turns out to be useful later.

\begin{lemma} \label{lmrfmc}
	With the notation as above, we have
	\begin{equation*}
		\exp \left( - {\sum_{i=1}^k \mathrm{FS}_i (H_i)} \right) d \mu'_0 = d \mu_{\mathrm{FS} (\bm{H})}.
	\end{equation*}
	In particular,
	\begin{equation*}
		\mathscr{L}^{\mathrm{coupled}} (\mathrm{FS}_1 (H_1) , \dots , \mathrm{FS}_k (H_k)) = \mathscr{L} (\mathrm{FS} (\bm{H})).
	\end{equation*}
\end{lemma}
\begin{proof}
First note that there exists $\varphi \in C^{\infty} (X , \rl)$ such that $e^{- \varphi } h_0 = h'_0$, or equivalently $e^{- \varphi /m} d \mu_0 = d \mu'_0$ in terms of the volume form. We then get the claim by the following straightforward computation that follows from the definitions (\ref{eqfscpdi}) and (\ref{eqfscpd}):
\begin{align*}
		\exp \left( - \sum_{i=1}^k \mathrm{FS}_i (H_i) \right) d \mu'_0 &=  \prod_{i=1}^k \left( \sum_{j=1}^{N_{i,m}} \left\Vert  s^{H_i}_{j_i} \right\Vert^2_{h_{i,0}} \right)^{-1/m}  d \mu'_0 \\
		&= \left( \sum_{j_{1} =1}^{N_{1,m}} \cdots \sum_{j_k =1}^{N_{k,m}} \left\Vert  s^{H_i}_{j_{1}} \otimes \cdots \otimes  s^{H_k}_{j_k} \right\Vert^2_{h_{1,0} \otimes \cdots \otimes h_{k,0}} \right)^{-1/m} e^{- \varphi / m } d \mu_0 \\
		&=  \left( \sum_{i=1}^k \sum_{j_{i}=1}^{N_{i,m}} \left\Vert  \bm{s}^{\bm{H}}_{\bm{j}} \right\Vert^2_{h_0} \right)^{-1/m} d \mu_0 \\
		&= \exp \left( - \mathrm{FS} (\bm{H}) \right) d \mu_0 = d \mu_{\mathrm{FS} (\bm{H})} , 
\end{align*}
where in the third line we used $e^{\varphi} h'_0 = e^{\varphi} h_{1,0} \otimes \cdots \otimes h_{k,0} = h_0$ and the last equality is the definitional (\ref{eqcmvf}).
\end{proof}

We can now define the following ``quantised'' functional.

\begin{definition}
The \textbf{quantised coupled Ding functional} $\mathscr{D}^{\mathrm{coupled}}_m : \bm{\CB}_m \to \rl$ is defined for $(H_1 , \dots , H_k) \in \bm{\CB}_{m}$ as
\begin{equation*}
\mathscr{D}^{\mathrm{coupled}}_m (H_1 , \dots , H_k) := \mathscr{L} (\mathrm{FS}(\bm{H})) - \sum_{i=1}^k \mathscr{E}_{i,m} (H_i),
\end{equation*}
where
\begin{equation*}
	\mathscr{E}_{i,m} (H_i) := -  \frac{1}{m N_{i,m}} \log \det (H_i \cdot H^{-1}_{i, 0}) .
\end{equation*}
A \textbf{coupled anticanonically balanced metric} at level $m$ is a critical point of $\mathscr{D}^{\mathrm{coupled}}_m$.
\end{definition}

The original definition given by Takahashi \cite[\S 3.1.2]{tak19} was 
\begin{equation*}
\mathscr{D}^{\mathrm{coupled}}_m (H_1 , \dots , H_k) := \mathscr{L}^{\mathrm{coupled}} (\mathrm{FS}_1 (H_1) , \dots , \mathrm{FS}_k (H_k)) - \sum_{i=1}^k \mathscr{E}_{i,m} (H_i),
\end{equation*}
but Lemma \ref{lmrfmc} ensures that it is exactly the same as the one above. We can also write down the ``zero of the moment map'' definition for the coupled anticanonically balanced metric, but for the detailed discussions we refer the reader to \cite[Proof of Proposition 3.4 and also Definition 3.1]{tak19}; the condition for the Bergman function, analogously to Proposition \ref{ppacbal}, can be established equally easily (cf.~\cite[\S 4.1]{tak19}).

Again as pointed out for $\mathscr{D}_m$ in \S \ref{scacbalm}, a critical point of $\mathscr{D}^{\mathrm{coupled}}_m$ is necessarily the global minimum over $\bm{\CB}_m$.

\section{Algebro-geometric preliminaries} \label{scagprem}

\subsection{Test configurations} \label{sctstcf}

We recall the basics of test configurations that will be used later.

\begin{definition}
	A \textbf{very ample test configuration $(\CX , \CL)$ of exponent $m$} for a Fano manifold $(X , -K_X)$ is a scheme $\CX$ endowed with a flat projective morphism $\pi : \CX \to \cx$, which is $\cx^*$-equivariant with respect to the natural $\cx^*$-action on $\cx$, with a relatively very ample Cartier divisor $\CL$ to which the action $\cx^* \actson \CX$ linearises, such that $\pi^{-1} (1) \cong X $ and $\CL |_{\pi^{-1} (1)} \cong -mK_X$. The preimage of $0 \in \cx$, written $\CX_0 := \pi^{-1} (0)$, is called the \textbf{central fibre}.
	
	We say that $(\CX , \CL)$ is \textbf{product} if $\CX$ is isomorphic to $X \times \cx$, and \textbf{trivial} if it is $\cx^*$-equivariantly isomorphic to $X \times \cx$ (i.e.~$\CX \isom X \times \cx$ and $\cx^*$ acts trivially on $X$).
\end{definition}

\begin{remark}
	It is not common in the literature to assume $\CL$ to be relatively very ample, and often important to consider the case when $\CL$ is merely semiample on $\CX$. In this paper, however, we restrict to very ample test configurations unless otherwise stated, since all test configurations of interest arise as a subscheme of a fixed projective space. The notion of the trivial test configuration as stated above is also rather uncommon recently, but the one above turns out to be the appropriate definition for this paper; see Remark \ref{rmtrtc}.
\end{remark}

The reason for considering very ample test configurations is that we can write down its ``matrix generator'', which is defined as the generator of the one-parameter subgroup associated to the $\cx^*$-action; throughout the text, we decree that a one-parameter subgroup is always algebraic, i.e.~a morphism between algebraic groups. More precisely, we have the following result by Ross--Thomas \cite{rt07} which is important later in connecting differential-geometric and algebro-geometric arguments.

\begin{proposition} \emph{(see \cite[Proposition 3.7]{rt07} and also \cite[\S 2.3]{bhj1})} \label{pprt}
	Let $(\CX , \CL)$ be a very ample test configuration of exponent $m$. Then there exists a one-parameter subgroup $\cx^* \to GL(H^0(X , -mK_X))$, with the generator $A \in \mathfrak{gl} (H^0 (X , -mK_X))$, such that $(\CX , \CL)$ can be realised as the Zariski closure of the orbit of $\iota(X) \subset \prj (H^0 (X , -mK_X)^{\vee} )$ under the $\cx^*$-action given by $A$; in other words $\CX$ is isomorphic to
	\begin{equation} \label{dftca}
		\CX_A := \overline{\{ ( \tau^{A} \cdot \iota (x) , \tau ) \mid x \in X \}_{\tau \in \cx^*}} \subset \prj (H^0 (X , -mK_X)^{\vee} ) \times \cx
	\end{equation}
	where the bar denotes the Zariski closure, $\pi: \CX_A \to \cx$ is the second projection, and the polarisation (denoted by $\CL_A$) is given by the restriction of the hyperplane bundle.
	
	Conversely, the Zariski closure $\CX_A$ as above, with the polarisation $\CL_A$, defines a very ample test configuration of exponent $m$ if $A \in \mathfrak{gl} (H^0 (X , -mK_X))$ is diagonalisable with integral eigenvalues.
\end{proposition}

Note that the endomorphism $A \in \mathfrak{gl} (H^0 (X , -mK_X))$ in (\ref{dftca}) is exactly the generator of the $\cx^*$-action on $H^0 (\CX_0 , \CL|_{\CX_0})$, as we can see from the proof of \cite[Proposition 3.7]{rt07} (see also \cite[Proposition 1.3 and \S 2.3]{bhj1}). Note also that $A$ is required to be diagonalisable with integral eigenvalues so that the one-parameter subgroup $\cx^* \ni \tau \mapsto \tau^{A} \in GL( H^0 (X , -mK_X)^{\vee})$ is a morphism of algebraic groups (rather than complex Lie groups), which is necessary for us to stay in the category of varieties and schemes.

Finally, we note that we can compactify a test configuration to form a family over $\prj^1$ (see also \cite[\S 2.2]{bhj1}).

\begin{definition}
	Let $(\CX , \CL)$ be a very ample test configuration for a Fano manifold $(X , -K_X)$. The \textbf{compactification} $\bar{\CX}$ of $\CX$ is defined by gluing together $\CX$ and $X \times (\prj^1 \setminus \{ 0 \})$ along their respective open subsets $\CX \setminus \CX_0$ and $X \times (\mathbb{C} \setminus \{ 0 \} )$, identified by the canonical $\cx^*$-equivariant isomorphism $\CX \setminus \CX_0 \isom X \times (\mathbb{C} \setminus \{ 0 \} )$. The line bundle $\bar{\CL}$ is a natural line bundle over $\bar{\CX}$ constructed from $\CL$ and the procedure for the compactification as above.
\end{definition}

\subsection{Ding invariant and Chow weight} \label{scdstcst}

We collect some definitions that are standard in the literature; see \cite[Definition 3.1]{Fujita18}, \cite[Definition 2.3]{Fujita2019}, and also \cite[\S 3]{Berman16} for more details on the following.


\begin{definition} \label{dfding}
	Let $(\CX , \CL)$ be a very ample test configuration for $(X , -K_X)$ of exponent $m$, and $\nu : \CX^{\nu} \to \CX$ be its normalisation with $\CL^{\nu} := \nu^* \CL$, whose compactification over $\prj^1$ is written as $(\bar{\CX}^{\nu} , \bar{\CL}^{\nu})$. Let $D_{(\CX^{\nu} , \CL^{\nu})}$ be a $\rtn$-divisor on $\CX^{\nu}$, whose support is contained in $\CX^{\nu}_0 := (\nu \circ \pi )^{-1} (0)$, such that $-m (K_{\bar{\CX}^{\nu} / \prj^1 } + D_{(\CX^{\nu} , \CL^{\nu})})$ is a Cartier divisor corresponding to $\bar{\CL}^{\nu}$; it is well-known that such a $\rtn$-divisor $D_{(\CX^{\nu} , \CL^{\nu})}$ exists uniquely. The \textbf{Ding invariant} of $(\CX , \CL)$ is a real number defined by
	 \begin{equation*}
	 	\mathrm{Ding} (\CX , \CL) := - \frac{(\bar{\CL}^{\nu})^{n+1}}{(n+1) m^{n+1} \mathrm{Vol}(X) } - 1 + \mathrm{lct} (\CX^{\nu} , D_{(\CX^{\nu} , \CL^{\nu})} ; \CX^{\nu}_0),
	 \end{equation*}
	 where $(\bar{\CL}^{\nu})^{n+1}$ stands for the intersection product over $\bar{\CX}^{\nu}$, and $\mathrm{lct} (\CX^{\nu} , D_{(\CX^{\nu} , \CL^{\nu})} ; \CX^{\nu}_0)$ is the log canonical threshold of $\CX^{\nu}_0$ with respect to $(\CX^{\nu} , D_{(\CX^{\nu} , \CL^{\nu})})$, as in Definition \ref{dflct}, noting that $\CX_0^{\nu}$ is an effective Cartier divisor on $\CX^{\nu}$ by the flatness of $\pi$.
\end{definition}

Note that there are several alternative definitions for the Ding invariant, including the version which is given in terms of the infimum over valuations; see \cite{BBJ,bhj1,bhj2} for more details.

\begin{definition} \label{dfchowwt}
	Let $(\CX , \CL)$ be a very ample test configuration for $(X , -K_X)$ of exponent $m$, which we realise as the Zariski closure of the orbit of $\iota(X) \subset \prj (H^0 (X , -mK_X)^{\vee} )$ under the $\cx^*$-action generated by $A \in \mathfrak{gl} (H^0 (X , -mK_X))$ by Proposition \ref{pprt}. The \textbf{Chow weight} is a rational number defined by
	\begin{equation*}
		\mathrm{Chow}_m (\CX , \CL) :=  \frac{(\bar{\CL}^{\nu})^{n+1}}{(n+1) m^{n+1} \mathrm{Vol}(X)} - \frac{\mathrm{tr} (A+A^*)}{m h^0 (X , -mK_X)} .
	\end{equation*}
\end{definition}

Ding polystability and Chow polystability can be defined by the nonnegativity of the Ding invariant and the Chow weight respectively, with equality if and only if the test configuration is product. We decide not to elaborate on the details since they will not be used in what follows.

\begin{remark} \label{rmdfdgch}
	Both these invariants are tightly connected to the Donaldson--Futaki invariant \cite[\S 2.1]{dontoric}. Indeed, Berman \cite[\S 3.1 and Corollary 3.9]{Berman16} proved that the Ding invariant agrees with the Donaldson--Futaki invariant for special test configurations. It is also well-known (see \cite{dontoric,rt07}) that the limit of the Chow weight is the Donaldson--Futaki invariant in the sense that
\begin{equation*}
	\mathrm{DF} (\CX , \CL) = \lim_{k \to \infty} mk \cdot \mathrm{Chow}_{mk} (\CX , k \CL) , 
\end{equation*}
where we note that $(\CX , k \CL)$ is a very ample test configuration for $(X , -K_X)$ of exponent $mk$, if $(\CX , \CL)$ is of exponent $m$.
\end{remark}


We recall the following stability condition defined by Saito--Takahashi \cite[Definition 3.4]{st19}.

\begin{definition} \label{dffstst}
A Fano manifold $(X , -K_X)$ is said to be \textbf{$F$-semistable} at level $m$ if for any very ample test configuration $(\CX , \CL)$ of exponent $m$ for $(X,-K_X)$ we have
\begin{equation*}
\mathrm{Ding} (\CX , \CL ) + \mathrm{Chow}_m  (\CX , \CL) \ge 0.
\end{equation*}
$(X, -K_X)$ is \textbf{$F$-stable} if it is $F$-semistable and the equality holds if and only if $(\CX , \CL)$ is trivial, and \textbf{$F$-polystable} if it is $F$-semistable with equality if and only if $(\CX , \CL)$ is product.
\end{definition}

\begin{remark} \label{rmtrtc}
	In the above, we recall that a test configuration $( \CX , \CL)$ was defined to be trivial if it is $\cx^*$-equivariantly isomorphic to $(X, -K_X) \times \cx$; note that, in terms of matrix generators, this amounts to saying that $(\CX , \CL) = ( \CX_A , \CL_A)$ with $A$ being a constant multiple of the identity matrix.
	
	Recall that this is not exactly the same as requiring $J^{\mathrm{NA}} ( \CX , \CL ) = 0$ (or equivalently the minimum norm being zero), which is another widely used definition for the trivial test configurations, because of the phenomenon first observed by Li--Xu \cite{lx14}; see \cite{bhj1,dertwisted} for more details.
\end{remark}

Definition \ref{dffstst} is slightly different from the one defined in \cite[Definition 3.4]{st19}, where the Donaldson--Futaki invariant, instead of the Ding invariant, was used. Note also that the invariant $\mathrm{Ding} (\CX , \CL ) + \mathrm{Chow}_m  (\CX , \CL)$ agrees with the quantised Futaki invariant introduced by Berman--Witt Nystr\"om \cite[\S 4.4]{BWN14} if $(\CX, \CL)$ is a special test configuration; this is proved in \cite[Lemma 3.2]{st19}, where we recall \cite{Berman16} that the Ding invariant agrees with the Donaldson--Futaki invariant for the special test configurations (Remark \ref{rmdfdgch}). Finally, while the Ding invariant is more inherently defined for the non-Archimedean metrics \cite[\S 7.7]{bhj1}, and hence its value unchanged under the normalisation, the Chow weight and the Donaldson--Futaki invariant are not. It is well-known, however, that these invariants decrease under the normalisation; see \cite[Proposition 5.1]{rt07} for the Chow weight and \cite[Proposition 3.15]{bhj1} for the Donaldson--Futaki invariant. This point is also mentioned in \cite[Lemma 3.5]{st19}.

\begin{remark} \label{rmtsliv}
	All the invariants above are \textit{translation invariant}, in the sense that they remain unchanged when the linearisation of the $\cx^*$-action on $\CL$ is twisted by the character $\tau^c$ for some $c \in \itg$; note that the twist of the test configuration $(\CX , \CL)$ by $\tau^c$ is $(\CX , \CL + c \CX_0)$ \cite[page 763]{bhj1}. In the formalism of Proposition \ref{pprt}, this amounts to saying that the invariants for $(\CX_A , \CL_A)$ remain unchanged when we replace $A$ by $A + c \cdot \mathrm{id}_{N_m}$.
\end{remark}

\subsection{K\"ahler--Ricci $g$-solitons case}

Stability conditions for the K\"ahler--Ricci $g$-solitons must be modified appropriately, as we describe below. Let $T^{\cx}$ be an algebraic torus of automorphisms. We first need the definition below, following \cite[Definition 2.1]{sze2007}.

\begin{definition}
	Let $T^{\cx}$ be an algebraic torus contained in $\mathrm{Aut}_0 (X)$. A very ample test configuration $(\CX , \CL)$ for a Fano manifold $(X , -K_X)$ is said to be \textbf{$T^{\cx}$-equivariant} if the action $T^{\cx} \actson \pi^{-1} (1) = (X , -K_X)$ extends to $(\CX , \CL)$ in such a way that it commutes with the $\cx^*$-action of $(\CX , \CL)$.
\end{definition}

We observe the following, by noting that $\iota (X)$ is not contained in any proper linear subspace of $\prj (H^0 (X , -mK_X)^{\vee})$.

\begin{lemma} \label{lmtccmtrs}
	Let $(\CX , \CL)$ be a very ample test configuration for $(X , -K_X)$ of exponent $m$ which we can write as $(\CX_A , \CL_A)$ with some $A \in \mathfrak{gl} (H^0 (X , -mK_X))$ as in Proposition \ref{pprt}. Suppose that we identify $T^{\cx} \subset \mathrm{Aut}_0 (X)$ with a subgroup of $GL(H^0 (X , -mK_X))$ by (\ref{eqautffebd}). Then, $(\CX , \CL)$ is $T^{\cx}$-equivariant if and only if $A$ commutes with all elements in $T^{\cx}$.
\end{lemma}

The version of $F$-stability for the K\"ahler--Ricci $g$-solitons is as follows.

\begin{definition}
	Let $(\CX , \CL)$ be a $T^{\cx}$-equivariant very ample test configuration for $(X , -K_X)$ of exponent $m$, which we realise as the Zariski closure of the $\cx^*$-orbit of $\iota(X) \subset \prj (H^0 (X , -mK_X)^{\vee})$ generated by $A \in \mathfrak{gl} (H^0 (X , -mK_X))$ that commutes with $T^{\cx}$ as in the above lemma. We define an invariant
	\begin{equation*}
		\mathscr{D}^{g , \mathrm{NA}}_m (\CX , \CL) := - 1 + \mathrm{lct} (\CX^{\nu} , D_{(\CX^{\nu} , \CL^{\nu})} ; \CX^{\nu}_0) - \frac{1}{m N_m \overline{g_m}} \sum_{\lambda \in P_m} g ( \lambda / m) \mathrm{tr} \left( (A+A^*)|_{R_{m , \lambda}} \right)  .
	\end{equation*}
	We say that $(X,-K_X)$ is \textbf{$F$-semistable in the sense of $g$-solitons} at level $m$ if for any $T^{\cx}$-equivariant very ample test configuration $(\CX , \CL)$ of exponent $m$ for $(X,-K_X)$ we have
\begin{equation*}
\mathscr{D}^{g , \mathrm{NA}}_m (\CX , \CL) \ge 0.
\end{equation*}
$F$-stability and $F$-polystability in the sense of $g$-solitons can be defined entirely analogously to Definition \ref{dffstst}.
\end{definition}

Note that the second term of $\mathscr{D}^{g , \mathrm{NA}}_m$ is precisely the slope of $-\mathscr{E}^g_m$ along Bergman geodesic rays by (\ref{eqslegm}).

\begin{remark}
	The following point was communicated to the author by Ryosuke Takahashi. Han--Li \cite[Definition 5.3]{HanLi20} defined an invariant $\mathscr{E}^{g , \mathrm{NA}} (\CX , \CL)$ for any very ample test configuration (more precisely they defined it for any positive non-Archimedean metric). Using this invariant and defining
	\begin{equation*}
		\mathrm{Ding}^g (\CX , \CL) :=  - 1 + \mathrm{lct} (\CX^{\nu} , D_{(\CX^{\nu} , \CL^{\nu})} ; \CX^{\nu}_0) - \mathscr{E}^{g , \mathrm{NA}} (\CX , \CL),
	\end{equation*}
	and
	\begin{equation*}
		\mathrm{Chow}^g_m (\CX , \CL) := \mathscr{E}^{g , \mathrm{NA}} (\CX , \CL) - \frac{1}{m N_m \overline{g_m}} \sum_{\lambda \in P_m} g ( \lambda / m) \mathrm{tr} \left( (A+A^*)|_{R_{m , \lambda}} \right),
	\end{equation*}
	for a very ample test configuration $(\CX , \CL)$ of exponent $m$, just as we defined the Chow weight in Definition \ref{dfchowwt}, we may write
	\begin{equation*}
		\mathscr{D}^{g , \mathrm{NA}}_m (\CX , \CL) = \mathrm{Ding}^g (\CX , \CL) + \mathrm{Chow}^g_m (\CX , \CL)
	\end{equation*}
	analogously to the strategy in \cite{st19}. This could be a more natural invariant since $\mathscr{E}^g_m$ is related to the ``quantisation'' of $\mathscr{E}^g$ (defined in (\ref{eqgmaeg})) by \cite[Proposition 4.5]{BWN14}, and $\mathscr{E}^{g , \mathrm{NA}}$ can be identified with the asymptotic slope of the functional $\mathscr{E}^g$ as
	\begin{equation*}
		\lim_{t \to + \infty} \frac{\mathscr{E}^g (\mathrm{FS (H_t)})}{t} = \mathscr{E}^{g , \mathrm{NA}} (\CX_A , \CL_A)
	\end{equation*}
	for the Bergman geodesic ray, by \cite[Proposition 5.8]{HanLi20}, analogously to the results in Theorem \ref{thasfed} which we review later.
\end{remark}

\subsection{Coupled Ding invariant} \label{sccdinv}

For the coupled \ke metrics, the relevant stability condition was introduced by Hultgren--Witt Nystr\"om \cite[Definition 1.14]{HWN19}. It was adapted to the balanced metrics by Takahashi \cite[Definition 3.10]{tak19} who defined the coupled version of the $F$-polystability. Takahashi's definition, however, does not seem strong enough to be equivalent to the existence of coupled anticanonically balanced metrics, as suggested by the argument in \S \ref{sceckem}. Thus we define a more stringent version of stability, which turns out to involve modifying the ``coupled'' test configurations defined by Hultgren--Witt Nystr\"om \cite[Definition 1.11]{HWN19}; see Remark \ref{rmctcod} for more details on the comparison with their original definition.

\begin{definition} \label{dftcgcp}
	Let $(\CX_i , \CL_i)$ be a very ample test configuration for $(X , L_i)$ of exponent $m$, for $i=1 , \dots , k$, defined as the Zariski closure of $\iota_i (X) \subset \prj(H^0 (X , m L_i)^{\vee})$ under the one-parameter subgroup $\tau^{A_i }$ generated by $A_i \in \mathfrak{gl} (H^0 (X, m L_i))$, as in Proposition \ref{pprt}. Recall also that we have the embedding
	\begin{equation*}
		\iota_{\mathrm{coupled}} : X \inj \prj (H^0 (X , m L_1 )^{\vee} \otimes \cdots \otimes  H^0 (X , m L_k)^{\vee}) =: \prj
	\end{equation*}
	by (\ref{mlmp}). We say that a very ample test configuration $(\CY , \CL_{\CY})$ is \textbf{generated by the $\cx^*$-actions of $(\CX_i , \CL_i)_{i=1}^k$}, if $\CY$ is defined as the Zariski closure of $\iota_{\mathrm{coupled}} (X)$ in $\prj$, with the reduced scheme structure, under the natural (dual) tensor product action of the one-parameter subgroup
	\begin{equation*}
	\tau^{A_1} \otimes \cdots \otimes \tau^{A_k }
	\end{equation*}
	on $H^0 (X , m L_1 )^{\vee} \otimes \cdots \otimes  H^0 (X , m L_k)^{\vee}$, and $\CL_{\CY} := \CO_{\prj} (1)|_{\CY}$. Note that $(\CY , \CL_{\CY})$ is a very ample test configuration for $(X , -K_X)$ of exponent $m$ since $- m K_X = \iota^*_{\mathrm{coupled}} \CO_{\prj} (1)$.
\end{definition}

The author believes that there should be a better way of formulating the above test configuration $(\CY , \CL_{\CY})$, hopefully in terms of non-Archimedean metrics. For example, it would be good if we can define what it means for a non-Archimedean metric $\phi^{\mathrm{NA}}_{\CY}$ on $-K_X$ to be generated by the $\cx^*$-actions of non-Archimedean metrics $\phi^{\mathrm{NA}}_1 , \dots , \phi^{\mathrm{NA}}_k$ on $L_1 , \dots , L_k$. This does not seem obvious at all from the above definition, since a priori $\CY$ depends on a particular choice of very ample test configurations and even their exponents: $\CY$ may change when we change $(\CX_i , \CL_i)$ to $(\CX_i , m' \CL_i)$ for $m' \in \mathbb{N}$ and for each $i=1 , \dots , k$. Moreover, $ (\CX_1 , \CL_1) , \dots , (\CX_k , \CL_k )$ all being product does not seem to necessarily imply that $(\CY , \CL_{\CY})$ is product unless these test configurations are all generated by the same holomorphic vector field, while it is easy to see that $(\CY , \CL_{\CY})$ is trivial if and only if $ (\CX_1 , \CL_1) , \dots , (\CX_k , \CL_k )$ are all trivial (in the sense of Remark \ref{rmtrtc}). Still, Definition \ref{dftcgcp} suffices for the purposes in this paper in which all test configurations arise as a subscheme of a fixed projective space.

\begin{definition} \label{dfcpdginv}
	Suppose that $(\CX_1 , \CL_1), \dots , (\CX_k , \CL_k )$ is a $k$-tuple of very ample test configurations, respectively for $(X, L_1) , \dots , (X , L_k)$, each of exponent $m$. Let $(\CY , \CL_{\CY})$ be a very ample test configuration generated by the $\cx^*$-actions of $(\CX_i , \CL_i)_{i=1}^k$. Let $\nu : \CY^{\nu} \to \CY$ be the normalisation of $\CY$ and $\CL_{\CY^{\nu}} := \nu^* \CL_{\CY}$. We define the \textbf{coupled Ding invariant} by
	\begin{equation*}
		\mathrm{Ding} \left( (\CX_i , \CL_i)_{i=1}^k \right) := - \sum_{i=1}^k \frac{(\bar{\CL}_i)^{n+1}}{(n+1) m^{n+1} \int_X c_1 (L_i)^n} - 1 + \mathrm{lct} (\CY^{\nu} , D_{\CY^{\nu} , \CL_{\CY^{\nu}}} ; \CY^{\nu}_0),
	\end{equation*}
	where $\mathrm{lct} (\CY^{\nu} , D_{(\CY^{\nu} , \CL_{\CY^{\nu}})} ; \CY^{\nu}_0)$ is the log canonical threshold of $\CY^{\nu}_0$ with respect to $(\CY^{\nu} , D_{(\CY^{\nu} , \CL_{\CY^{\nu}})})$, which is defined exactly as in Definitions \ref{dfding} and \ref{dflct}.
\end{definition}

It turns out that this invariant arises as the asymptotic slope of the coupled Ding functional (Theorem \ref{ppcdgas}). It is natural to define a stability condition as follows.

\begin{definition} \label{dfuncds}
	A Fano manifold $(X, -K_X)$ with the decomposition $-K_X = L_1 + \cdots + L_k$ is said to be \textbf{coupled Ding stable} if for any $m \in \mathbb{N}$ and any $k$-tuple of very ample test configurations $(\CX_1 , \CL_1), \dots , (\CX_k , \CL_k )$ respectively for $(X, L_1) , \dots , (X , L_k)$, each of exponent $m$, we have
	\begin{equation*}
		\mathrm{Ding} \left( (\CX_i , \CL_i)_{i=1}^k \right) \ge 0
	\end{equation*}
	with equality if and only if $(\CX_i , \CL_i)$ is trivial for each $i=1 , \dots , k$.
\end{definition}

The appropriate condition for the equality case could be $J^{\mathrm{NA}} (\CX_i , \CL_i) = 0$ for all $i=1 , \dots , k$ (cf.~Remark \ref{rmtrtc}), or may even be $J^{\mathrm{NA}} (\CY , \CL_{\CY}) = 0$, but we decide not to discuss this stability condition any further as it will not be used in the rest of this paper. The one that we use in this paper is its $F$-stability version, which will be stated later in Theorem \ref{thckebm} and Corollary \ref{crcpdm}.

\begin{remark} \label{rmctcod}
	In the original definition \cite[Definition 1.10]{HWN19} by Hultgren--Witt Nystr\"om (and also its $F$-stability version of Takahashi \cite[Definition 3.10]{tak19}), they further assume $\CY = \CX_1 = \cdots = \CX_k$ and $\CL_{\CY} = \sum_{i=1}^k \CL_i$ as a definition of the coupled test configurations, which may be too stringent (hence leading to a weaker stability condition) as suggested by the computation for the balanced metrics that we provide in \S \ref{sceckem}. They also assume that $\CY = \CX_1 = \cdots = \CX_k$ is a normal $\rtn$-Gorenstein variety and consider the $K$-stability as opposed to the Ding stability, but this seems to be a minor difference. While the author believes that the stability condition in Definition \ref{dfuncds}, or its appropriately modified version, is a useful one in studying the coupled \ke metrics, it is important to note that none of the results in \cite{DelHul,FutZha2019,FutZha2020,Hultgren,HWN19,Nakamura,tak19} seems to be affected when we adopt Definition \ref{dfuncds} as the relevant notion of stability. The works \cite{DelHul,FutZha2019,FutZha2020,Hultgren,Nakamura}, from the point of view of stability, essentially consider the case when $(\CX_1 , \CL_1),  \dots , (\CX_k , \CL_k )$ are all product and generated by the same holomorphic vector field, and hence we have $\CY = \CX_1 = \cdots = \CX_k = X \times \cx$; \cite{DelHul,Hultgren} also consider multiple holomorphic vector fields but in this case the equation in question seems different. Definition \ref{dfuncds} does not affect the result in \cite{HWN19} (resp.~\cite{tak19}) which proves that the existence of coupled \ke metrics implies their weaker version of coupled $K$-stability (resp.~its $F$-stability version \cite[\S 3.2]{tak19}). It also seems interesting to compare the above definition with the version for the constant scalar curvature K\"ahler metrics in \cite{DatPin}.
\end{remark}

Given the above definition for the coupled Ding invariant, it is natural to consider the uniform coupled Ding stability, following the definition of \cite{bhj1,dertwisted} (and also \cite{szethesis}), and the $G$-uniform coupled Ding stability when the automorphism group is nondiscrete \cite{Hisamotosl,Li19G}. Moreover, given recent results for the K\"ahler--Einstein metrics and the K\"ahler--Ricci $g$-solitons \cite{BBJ,HanLi20,Hisamotomab,Li19G}, it seems natural to expect that the $G$-uniform coupled Ding stability is equivalent to the existence of coupled K\"ahler--Einstein metrics. We decide not to discuss these problems any further in this paper; related results may appear elsewhere.

\section{Psh rays, convexity, and the slope formulae} \label{scprcsf}

We now recall several foundational results that connect the asymptotic slope of the functionals $\mathscr{D}, \mathscr{L}, \mathscr{E}$ to the algebro-geometric invariants defined above. We start by recalling the following facts that are surely well-known to the experts.

\begin{lemma} \label{lmbgnam}
The following hold for the Bergman geodesic ray (\ref{eqfst}).
\begin{enumerate}
\renewcommand{\labelenumi}{(\roman{enumi})}
	\item The Bergman geodesic ray is a psh (plurisubharmonic) ray of linear growth.
	\item Suppose that $A \in \mathfrak{gl} (H^0 (X , -mK_X))$ in (\ref{eqfst}) is $H_0$-hermitian with integral eigenvalues. Writing $\lambda_{\mathrm{min}} \in \mathbb{Z}$ for the minimum eigenvalue of $A$, the Bergman geodesic ray defined by $A' := A - \lambda_{\mathrm{min}} \cdot \mathrm{id}_{N_m}$ admits $\phi_{A'}^{\mathrm{NA}}$ as non-Archimedean limit, where $\phi_{A'}^{\mathrm{NA}} \in \CH^{\mathrm{NA}}$ is the positive non-Archimedean metric on $-K_X$ represented by the very ample test configuration $(\CX_{A'} , \CL_{A'})$ of exponent $m$.
\end{enumerate}
\end{lemma}

In the above, a psh ray is as defined in \cite[Definition 1.3]{BBJ}, which is also called a ``subgeodesic'' in \cite[\S 3.1]{bhj2} and many other papers. Recall also that a positive non-Archimedean metric on $-K_X$ is a certain equivalence class of semiample test configurations for $(X , -K_X)$ (see \cite[Definition 6.2]{bhj1}), and the notions of psh rays and non-Archimedean limits are as defined in \cite[Definition 3.1]{bhj2}. We refer the reader to \cite{bhj1,bhj2} for more details on non-Archimedean metrics and non-Archimedean limits, and do not give a detailed review of them in this paper as it is rather technical. The only remark here is that the re-scaling $A \mapsto A - \lambda_{\mathrm{max}} \cdot \mathrm{id}_{N_m}$ in (ii) does not change the total space of the test configuration and only results in a difference in linearising the $\cx^*$-action on the polarisation, as we saw in Remark \ref{rmtsliv}.

\begin{proof}
	Ascertaining both claims is an easy exercise in checking the definitions given in \cite[Definition 1.3 and (4.1)]{BBJ} for the first item, and \cite[\S 3.1]{bhj2} for the second. Writing $p_1 : X \times \cx^* \to X$ for the natural projection, the required semipositivity for the first item is a consequence of the well-known inequality $(p_1^*\omega_0 + \partial_{X \times \cx^*} \bar{\partial}_{X \times \cx^*} \mathrm{FS} (H_t))^{n+1} \ge 0$ which follows from e.g.~\cite[Proposition 3]{donsymm} and \cite[page 351]{donproj2}. To see that it is of linear growth, we only need to note, from (\ref{eqfst}),
	\begin{equation*}
	\lim_{t \to + \infty} \frac{1}{t} \sup_X \mathrm{FS} (H_t) = \lim_{t \to + \infty} \frac{1}{m} \log \left( \sup_X \sum_{i=1}^{N_m} \left\Vert e^{tA} s_i \right\Vert^2_{h_0} \right)^{1/t} \le \frac{\lambda_{\mathrm{max}}}{m}
	\end{equation*}
	where $\lambda_{\mathrm{max}}$ is the largest modulus of the eigenvalues of $A$.

For the second item, by replacing the reference basis $\{ s_i \}_{i=1}^{N_m}$ by an $H_0$-unitarily equivalent basis if necessary, we may diagonalise $A' = \mathrm{diag}(\lambda'_1 ,  \dots , \lambda'_{N_m} )$, $\lambda'_1 \ge  \cdots \ge \lambda'_{N_m} \ge 0$, and hence, for $H'_t := e^{-t(A')^*} e^{-tA'}$, we have
	\begin{equation*}
		\mathrm{FS} (H'_t) = \frac{1}{m} \log \left( \sum_{i=1}^{N_m} \left\Vert e^{\lambda'_i t} s_i \right\Vert^2_{h_0} \right) .
	\end{equation*}
	Then the metric $\exp (-m \mathrm{FS} (H'_t)) h_0^m$ indeed extends to a smooth metric on $\CL_{A'}$ over $\cx$ as $\tilde{h}_0 |_{\CX_{A'}}$, where we recall our notational convention (\ref{dfhtilde}); note that this metric may be degenerate on the central fibre of $\CX_{A'}$. Thus we get the statement \cite[Definition 3.1]{bhj2} that we needed to prove.
\end{proof}

We recall the convexity result due to Berndtsson \cite{ber15} which plays a very important role in the proof, which we can apply to the Bergman geodesic rays by Lemma \ref{lmbgnam} (see also \cite[Lemmas 6.5 and 7.2]{BBGZ}).

\begin{theorem} \emph{(Berndtsson \cite[Theorems 1.1 and 1.2]{ber15})} \label{thberndt}
	$\mathscr{L}$ is convex along the Bergman geodesic rays. More precisely, we have
	\begin{equation*}
		\frac{d^2}{dt^2} \mathscr{L} (\mathrm{FS} (H_t)) \ge 0
	\end{equation*}
	for the Bergman geodesic ray (\ref{eqfst}), and the equality holds only if there exists a holomorphic vector field on $X$ with the flow $\psi_t$ such that
	\begin{equation*}
	\psi^*_t (\ddbar \mathrm{FS} (H_t)) = \ddbar \mathrm{FS} (H_0) .
	\end{equation*}
\end{theorem}

We further recall the following foundational theorem, which is a collection of the results proved in \cite[Theorem 5.4]{BBJ}, \cite[Proposition 3.8]{Berman16}, \cite[Proposition 7.29]{bhj1}, and \cite[Theorems 3.6 and 3.7]{bhj2}, summarised for our purposes in this paper.


\begin{theorem} \emph{(see \cite{Berman16,BBJ,bhj1,bhj2})} \label{thasfed}
	Suppose that an $H_0$-hermitian matrix $A \in \mathfrak{gl} (H^0 (X , -mK_X))$ has rational eigenvalues, and choose $c \in \mathbb{N}$ so that $cA$ has integral eigenvalues. Writing $ H_t := e^{-tA^*} e^{tA}$ for the Bergman geodesic ray and $(\CX_{cA} , \CL_{cA})$ for the test configuration associated to $cA$ as in Proposition \ref{pprt}, we have
	\begin{equation*}
		\lim_{t \to + \infty} \frac{d}{dt} \mathscr{L} (\mathrm{FS} (H_t)) = \lim_{t \to + \infty} \frac{\mathscr{L} (\mathrm{FS} (H_t))}{t} = \frac{1}{c} \left( -1 + \mathrm{lct} (\CX_{cA}^{\nu} , D_{(\CX_{cA}^{\nu} , \CL_{cA}^{\nu})} ; \CX_{cA,0}^{\nu}) \right),
	\end{equation*}
	and
	\begin{equation*}
		\lim_{t \to + \infty} \frac{d}{dt} \mathscr{E} (\mathrm{FS} (H_t)) =  \lim_{t \to + \infty} \frac{\mathscr{E} (\mathrm{FS} (H_t))}{t} = \frac{1}{c} \frac{(\bar{\CL}_{cA}^{\nu})^{n+1}}{(n+1) m^{n+1} \mathrm{Vol}(X) } .
	\end{equation*}
	In particular,
	\begin{equation*}
		\lim_{t \to + \infty} \frac{d}{dt} \mathscr{D} (\mathrm{FS} (H_t)) = \lim_{t \to + \infty} \frac{\mathscr{D} (\mathrm{FS} (H_t))}{t} = \frac{1}{c} \mathrm{Ding} (\CX_{cA} , \CL_{cA}).
	\end{equation*}
\end{theorem}


The equality $\lim_{t \to + \infty} \frac{d}{dt} \mathscr{L} (\mathrm{FS} (H_t)) = \lim_{t \to + \infty} \mathscr{L} (\mathrm{FS} (H_t)) / t$ is a direct consequence of Berndtsson's convexity, and the corresponding statement for $\mathscr{E}$ follows from the well-known fact that $\mathscr{E}$ is convex along the Bergman geodesic rays (see \cite[Proposition 1]{donproj2}). We also observe
\begin{equation*}
	\lim_{t \to + \infty} \frac{\mathscr{L} (\mathrm{FS} (H_t))}{t} = \lim_{ct \to + \infty} \frac{1}{c} \frac{\mathscr{L} (\mathrm{FS} (H_{ct}))}{t} = \frac{1}{c} \left( -1 + \mathrm{lct} (\CX_{cA}^{\nu} , D_{(\CX_{cA}^{\nu} , \CL_{cA}^{\nu})} ; \CX_{cA,0}^{\nu}) \right)
\end{equation*}
for $c \in \mathbb{N}$, and similarly for $\mathscr{E}$ and $\mathscr{D}$. To get the same result, we can also consider the pullback by the base change $\cx \ni \tau \mapsto \tau^c \in \cx$ ($c \in \mathbb{N}$) as in \cite[\S 6.3]{bhj1}, combined with \cite[Proposition 3.5 and Theorem 3.7]{bhj2}. Note also that $\mathscr{L}$ and $\mathscr{E}$ are not translation invariant, but their difference $\mathscr{D}$ (and also $\mathscr{D}_m$) is (cf.~Remark \ref{rmdgtrinv}).


\begin{remark} \label{rmlenawdf}
	For an $H_0$-hermitian $A \in \mathfrak{gl} (H^0 (X , -mK_X))$, not necessarily of rational eigenvalues, we can show that the asymptotic slopes of the functionals $\mathscr{L}, \mathscr{E}$ are well-defined real numbers since the Bergman geodesic ray is a psh ray of linear growth by Lemma \ref{lmbgnam}; see \cite[Theorem 5.4]{BBJ} for the proof for $\mathscr{L}$, and e.g.~\cite[Proposition 4.1]{BBJ}, \cite[Proposition 4.1]{BDL17} for $\mathscr{E}$. In particular, the asymptotic slopes for $\mathscr{D}, \mathscr{D}_m$ are well-defined real numbers even when $A \in \mathfrak{gl} (H^0 (X , -mK_X))$ does not have rational eigenvalues.
\end{remark}

\section{Proof of the main results} \label{scpfmr}

\subsection{Proof of Theorem \ref{mthst} and Corollary \ref{mthstc}} \label{scpfmtmc}

We pick a basis for $H^0 (X , -mK_X)^{\vee}$ that is dual to the reference basis $\{ s_i \}_{i=1}^{N_m}$ in Remark \ref{rmrfbs}, and identify $\prj ( H^0 (X , -mK_X)^{\vee})$ with $\prj^{N_m -1}$ in what follows. We write $[Z_1 : \cdots : Z_{N_m}]$ for the homogeneous coordinates for $\prj^{N_m-1}$. The Fano manifold $X$ is considered to be embedded in $\prj^{N_m -1}$ by $\iota$, and we observe that the embedded variety $\iota (X)$ is not contained in any proper linear subspace of $\prj^{N_m-1}$.

Recall first that for any $A \in \mathfrak{gl} (H^0 (X , -mK_X))$, which will be assumed to be $H_0$-hermitian in what follows, and the embedded variety $\iota(X) \subset \prj^{N_m -1} $, the flat limit (or the limit in the Hilbert scheme)
\begin{equation} \label{dffllm}
	\CX_0 := \lim_{t \to + \infty} e^{-tA} \cdot \iota( X ) \subset \prj^{N_m -1}
\end{equation}
is a well-defined projective scheme in $\prj^{N_m -1}$ defined by, in a down-to-earth terminology, the homogeneous ideal generated by the limit of the defining homogeneous polynomials for $\iota ( X) \subset \prj^{N_m -1} $ (see e.g.~\cite[Definition 6.3]{szebook}). We write $\CX_{0, \mathrm{red}}$ for the reduced part of $\CX_0$, which, again in a down-to-earth terminology, is equal to $\CX_0$ as an algebraic set in $\prj^{N_m -1}$ but has the reduced scheme structure (note that $\CX_0$ can have multiple components). We also write $\CX^{\mathrm{reg}}_{0, \mathrm{red}}$ for the regular locus of the reduced part of $\CX_0$. 

\begin{remark}
	The family $\pi : \CX_{A} \to \cx$ defined as in (\ref{dftca}) above is not a test configuration if $A$ has non-integral eigenvalues, but is still flat (see \cite[\S 6.2]{szebook} for more details). In fact we do not need the flatness later, and we only need $\dim_{\cx} \CX^{\mathrm{reg}}_{0, \mathrm{red}} = n$ for the later argument. This result is also a consequence of \cite[proof of Lemma 3.15]{hk18} which constructs a test configuration that has the flat limit (\ref{dffllm}) as its central fibre.
\end{remark}

Recalling also the variable $\tau$ in \S \ref{sctstcf}, we have a change of variables
\begin{equation} \label{eqcvtaut}
	\tau = e^{-t} ,
\end{equation}
so that the limit of $\tau^{A} $ as $\tau \to 0$ corresponds to $t \to + \infty$; a subtlety may be that $\tau \in \cx^*$ while $t \in \rl$, but this does not affect the description of the flat limit up to a difference in unitarily rotating the embedding inside $\prj^{N_m -1}$, since the unitary group is compact (see also \cite[Lemma 2]{donlb}). This correspondence allows us to compare the asymptotic behaviour of the (differential-geometric) energy functionals and algebro-geometric invariants given in terms of the central fibre of test configurations.

For the $H_0$-hermitian endomorphism $A \in \mathfrak{gl} (H^0 (X , -mK_X))$ we perform a unitary change of basis for $H^0 (X , -mK_X)$, i.e.~replace the reference basis $\{ s_i \}_{i=1}^{N_m}$ (resp.~homogeneous coordinates $[Z_1 : \cdots : Z_{N_m}]$) by an $H_0$-unitarily equivalent one which is still written as $\{ s_i \}_{i=1}^{N_m}$ (resp.~$[Z_1 : \cdots : Z_{N_m}]$) by abuse of notation, so that $A = \mathrm{diag} (\lambda_1 , \dots , \lambda_{N_m})$ with $\lambda_1 \ge \cdots \ge \lambda_{N_m}$, with respect to this new basis. We first prove the following lemma.

\begin{lemma} \label{lmlmacb}
	Suppose that we write $A' := A - \lambda_{N_m} \cdot \mathrm{id}_{N_m}$ and consider the flat family $\pi : \CX_{A'} \to \cx$ as in (\ref{dftca}), whose central fibre is the flat limit $\CX_0$ defined by $A'$ as in (\ref{dffllm}). Writing $\CO (1)$ for the hyperplane bundle over $\prj^{N_m-1}$ which $\CX_0$ is embedded in, and $\tilde{h}_0$ for the Fubini--Study metric on $\CO (1)$ defined by the reference hermitian form $H_0 \in \CB_m$, the fibrewise $m$-th root $\tilde{h}_0^{1/m}$ defines a volume form on $\CX^{\mathrm{reg}}_{0, \mathrm{red}}$ which may be degenerate.
\end{lemma}

\begin{proof}
	Recalling the second item of Lemma \ref{lmbgnam}, we find that $\tilde{h}_0 |_{\CX_{A'}}$ is a smooth hermitian metric on $\CL_{A'} = \CO(1) |_{\CX_{A'}} $ which is allowed to be degenerate on $\CX_{0, \mathrm{red}}$. For any point $p \in \CX^{\mathrm{reg}}_{0, \mathrm{red}}$ we pick an open set $\CU \subset \CX_{A'}$ (in the Euclidean topology) with $(p,0) \in \CU \subset \prj^{N_m -1} \times \cx$.
	
	We observe that there exists a holomorphic coordinate system of $\CX^{\mathrm{reg}}_{0, \mathrm{red}}$ around $p$ which can be perturbed to give a holomorphic coordinate system of a point in a nearby noncentral fibre, as follows. We first pick a set of homogeneous polynomials defining $\CX^{\mathrm{reg}}_{0, \mathrm{red}}$ near $p$ and choose smooth coordinates by the implicit function theorem. By the definition of the flat limit we can pick a point $q	$ in $\pi^{-1} (t) \cong e^{-tA'} \cdot \iota (X)$ such that $t \neq 0$ and $q$ is close to $p$ in $\prj^{N_m -1} \times \cx$ in the Euclidean topology, and we can choose $t$ to be sufficiently close to $0$ so that the set of defining homogeneous polynomials remain nondegenerate to satisfy the hypothesis of the implicit function theorem. Then the parameter-dependent implicit function theorem (see e.g. \cite[Theorem 2.3]{Gloeckner}) implies that the holomorphic coordinates near $q$ thus chosen converge to the ones near $p$ at least continuously with respect to the Euclidean topology of $\prj^{N_m -1} \times \cx$.
		
	Writing $\CV := \CU \setminus \CX_{0, \mathrm{red}}$, we find that $\tilde{h}^{1/m}_0 |_{\CV}$ defines a smooth family of $2n$-forms on $\CV$, since at each point $(q, e^{-t}) \in \CV$ we have $\CO(1)|_{(q,e^{-t})} = -mK_{e^{-tA'} \cdot \iota(X)} |_q$ for the anticanonically embedded Fano manifold $e^{-tA'} \cdot \iota(X) \subset \prj^{N_m-1}$, by recalling the correspondence $\mathcal{H}om_{C^{\infty}_X} ((-K_X) \otimes \overline{(-K_X)} , \cx ) = K_X \otimes \overline{K_X}$ explained in \S \ref{sckemdf}. Since $\tilde{h}_0 |_{\CX_{A'}}$ is a smooth hermitian metric over $\CX_{A'}$ we may extend the positive $2n$-form $\tilde{h}^{1/m}_0 |_{\CV}$ to a semipositive $2n$-form on $\CU$ by continuity, noting that the resulting form on $\CU \cap \CX^{\mathrm{reg}}_{0, \mathrm{red}}$ can be written explicitly in terms of the homogeneous coordinates $[Z_1 : \cdots : Z_{N_m}]$ and the homogeneous ideal defining $\CX^{\mathrm{reg}}_{0, \mathrm{red}}$, or equivalently the homogeneous coordinate system that can be perturbed to the one for the nearby point, as discussed above. Finally, we have $\dim_{\cx} \CX^{\mathrm{reg}}_{0, \mathrm{red}} = n$ by the flatness of $\pi : \CX_{A'} \to \cx$, and hence a semipositive $2n$-form on $\CX^{\mathrm{reg}}_{0, \mathrm{red}}$ defines a volume form which may be degenerate.
\end{proof}


The key proposition, which we prove by adapting the strategy of \cite[\S 4]{donlb}, is the following.

\begin{proposition} \label{lmasll}
	Let $A \in \mathfrak{gl} (H^0 (X , -mK_X))$ be $H_0$-hermitian with eigenvalues $\lambda_1 \ge \cdots \ge \lambda_{N_m}$. We have
	\begin{equation*}
		\lim_{t \to + \infty} \frac{d}{dt} \mathscr{L} (\mathrm{FS}(H_t)) = \frac{2}{m} \sum_{i=1}^{N_m} \lambda_i C_i (\CX_0),
	\end{equation*}
	for some real numbers $C_1 (\CX_0) , \dots , C_{N_m} (\CX_0)$, with $\sum_{i=1}^{N_m} C_i (\CX_0) = 1$, which depends only on the flat limit $\CX_0$ as defined in (\ref{dffllm}) and the $H_0$-orthonormal basis $\{ s_i \}_{i=1}^{N_m}$ that diagonalises $A$.
\end{proposition}

Recall that by Remark \ref{rmlenawdf} we already know that the above asymptotic slope is a well-defined real number.

\begin{proof}
	We consider the embedded variety $e^{-t A} \cdot \iota (X) \subset \prj^{N_m -1}$ as before. As in the previous lemma, $\tilde{h}_0 |_{e^{-t A} \cdot \iota (X)}$ defines a hermitian metric on the line bundle $-m K_{e^{-t A} \cdot \iota (X)}$ and hence a volume form $ \tilde{h}_0^{1/m} |_{e^{-t A} \cdot \iota (X)}$ on $e^{-tA} \cdot \iota (X)$. With this understood, we can re-write Lemma \ref{lmdvffs} as
\begin{align*}
		\frac{d}{dt} \mathscr{L} (\mathrm{FS}(H_t)) &= \frac{2}{m} \sum_{i=1}^{N_m} \lambda_i  \left(  \int_{e^{-tA} \cdot \iota (X)}  \tilde{h}^{1/m}_0  \right)^{-1} \int_{e^{-tA} \cdot \iota (X)}  \frac{ \Vert Z_i \Vert^2_{\tilde{h}_0}}{\sum_{l=1}^{N_m} \Vert Z_l \Vert^2_{\tilde{h}_0}} \tilde{h}^{1/m}_0 \\
		&= \frac{2}{m} \sum_{i=1}^{N_m} \lambda_i  \left(  \int_{e^{-tA'} \cdot \iota (X)}  \tilde{h}^{1/m}_0  \right)^{-1} \int_{e^{-tA'} \cdot \iota (X)}  \frac{ \Vert Z_i \Vert^2_{\tilde{h}_0}}{\sum_{l=1}^{N_m} \Vert Z_l \Vert^2_{\tilde{h}_0}} \tilde{h}^{1/m}_0,
\end{align*}
where $A':= A - \lambda_{\mathrm{min}} \cdot \mathrm{id}_{N_m}$ for the least eigenvalue $\lambda_{\mathrm{min}} \in \rl$ of $A$. While the effect of re-scaling cancels out between the denominator and the numerator, it is chosen so that the hermitian metric $\exp (- m \mathrm{FS} (H'_t) ) h^m_0 = \tilde{h}_0 |_{e^{-tA'} \cdot \iota (X)}$ extends to a smooth metric over $\CX_{A'}$ which may be degenerate along the central fibre, as we saw in the proof of Lemmas \ref{lmbgnam} and \ref{lmlmacb}.

We take the limit of the domain of the integration, in the sense of the integral currents, as in \cite[\S 4]{donlb}. We write $|\CX_0|$ for the algebraic cycle defined by $\CX_0$, which we may write as $|\CX_0| = \sum_{j} \nu_j |V_j|$, using a finite $\itg$-linear combination of algebraic cycles defined by the reduced and irreducible components of $\CX_0$ ($\nu_j >0$). By Lemma \ref{lmlmacb}, $\tilde{h}_0^{1/m}$ defines a $2n$-form (and hence a volume form) on $\CX^{\mathrm{reg}}_{0, \mathrm{red}}$, which naturally defines an absolutely continuous measure on each reduced and irreducible component of $\CX_0$, since $\tilde{h}_0$ is a smooth (in particular bounded) metric over $\CX_{A'}$ and hence has no singular part that is supported on a Zariski closed subset. Thus, the definition of the flat limit (\ref{dffllm}) implies that we can write the limit $t \to + \infty$ of the above integral as
\begin{equation*}
		\lim_{t \to + \infty} \frac{d}{dt} \mathscr{L} (\mathrm{FS}(H_t)) = \frac{2}{m} \sum_{i=1}^{N_m} \lambda_i  \left( \sum_j \nu_j \int_{V_j \cap \CX^{\mathrm{reg}}_{0, \mathrm{red}}}  \tilde{h}^{1/m}_0  \right)^{-1} \left( \sum_{j} \nu_j \int_{V_j \cap \CX^{\mathrm{reg}}_{0, \mathrm{red}}}  \frac{ \Vert Z_i \Vert^2_{\tilde{h}_0}}{\sum_{l=1}^{N_m} \Vert Z_l \Vert^2_{\tilde{h}_0}} \tilde{h}^{1/m}_0 \right) .
\end{equation*}

We claim that each summand on the right hand side is a well-defined real number which depends only on the flat limit $\CX_0$ and the homogeneous coordinates $[Z_0 : \cdots : Z_{N_m} ]$. We first claim that the denominator is a well-defined nonzero real number. To prove this claim, pick the homogeneous coordinate $Z_{\mathrm{min}}$ corresponding to the eigenvector associated to the least eigenvalue $\lambda_{\mathrm{min}}$ of $A$. Since $\iota (X)$ is not contained in any proper linear subspace of $\prj^{N_m -1}$, there exists at least one point $p_1 \in \iota(X)$ where $Z_{\mathrm{min}}$ is nonzero at $p_1$. Taking the limit $\lim_{t \to + \infty} e^{-tA'} \cdot p_1 =:p_0$, which is a (well-defined) point in $\CX_{0, \mathrm{red}}$ that is fixed by the $\cx^*$-action induced by $A'$ (by the definition of the flat limit (\ref{dffllm})), there exists at least one point $p \in \CX^{\mathrm{reg}}_{0, \mathrm{red}}$ where $Z_{\mathrm{min}}$ is nonzero at $p$ by continuity. With the scaling of $A'$ as above, we find that the volume form $\tilde{h}_0^{1/m} |_{\CU}$ on some Euclidean open set $\CU$ in $\CX^{\mathrm{reg}}_{0, \mathrm{red}}$ containing $p$, given by Lemma \ref{lmlmacb}, is nondegenerate. Outside this Euclidean open set, $\tilde{h}_0^{1/m}$ remains bounded over the whole of $\CX^{\mathrm{reg}}_{0, \mathrm{red}}$ (an important point being that we allow it to be zero), again by the scaling that we chose for $A'$. By extending $\tilde{h}_0^{1/m}$ as an absolutely continuous measure over $\CX_{0, \mathrm{red}}$, we thus find that the total volume over $V_j \cap \CX^{\mathrm{reg}}_{0, \mathrm{red}}$ is finite for all $j$ and nonzero for some $j$, since $\CX_{0, \mathrm{red}}$ is compact in the Euclidean topology. It is immediate from the discussions so far that each summand depends only on the flat limit $\CX_0$, and not on the particular generator $A \in \mathfrak{gl} (H^0 (X , -mK_X))$ which defines $\CX_0$ (see also Remark \ref{rmtsliv}).

Thus we find that for each $i = 1 , \dots , N_m$,
\begin{equation*}
	C_i (\CX_0) := \left( \sum_j \nu_j \int_{V_j \cap \CX^{\mathrm{reg}}_{0, \mathrm{red}}}  \tilde{h}^{1/m}_0  \right)^{-1} \left( \sum_{j} \nu_j \int_{V_j \cap \CX^{\mathrm{reg}}_{0, \mathrm{red}}}  \frac{ \Vert Z_i \Vert^2_{\tilde{h}_0}}{\sum_{l=1}^{N_m} \Vert Z_l \Vert^2_{\tilde{h}_0}} \tilde{h}^{1/m}_0 \right)
\end{equation*}
is a well-defined real number depending only on the flat limit $\CX_0$ and the homogeneous coordinates $[Z_1 : \cdots : Z_{N_m}]$ which correspond to the reference basis $\{ s_i \}_{i=1}^{N_m}$, as required; it is obvious from the above that we have $\sum_{i=1}^{N_m} C_i (\CX_0) = 1$.
\end{proof}

We now begin the proof of Theorem \ref{mthst}. Lemma \ref{lmlmacb} and Proposition \ref{lmasll} give us the necessary ingredients to pursue the strategy introduced by Keller and the author for the $J$-balanced metrics in \cite[\S 3.2]{hk18}, while some additional arguments are necessary to deal with the nontrivial holomorphic vector fields.

\begin{proof}[Proof of Theorem \ref{mthst}]
Suppose that $(X , -K_X)$ is $F$-polystable at level $m$. We first show that
\begin{equation} \label{eqpfmthst1}
	\lim_{t \to + \infty} \frac{d}{dt} \mathscr{D}_m (H_t)  \ge 0 
\end{equation}
for all Bergman geodesic rays $\{ H_t \}_{t \ge 0} \subset \CB_m$, $H_t = e^{-tA^*} e^{-tA}$ as defined in (\ref{eqbggda}), with equality if and only if $A \in \mathfrak{aut} (X)_r$; recall that $\mathscr{D}_m$ is convex along $\{ H_t \}_{t \ge 0}$ by (\ref{eqemqaff}) and Theorem \ref{thberndt}.

Suppose that $A$ is $H_0$-hermitian and has rational eigenvalues. Following the idea of Saito--Takahashi \cite{st19}, we write $\mathscr{D}_m (H_t) = \mathscr{D} (\mathrm{FS} (H_t)) + \mathscr{E} (\mathrm{FS} (H_t)) - \mathscr{E}_m (H_t)$. Then Lemma \ref{lmdvffs}, Definition \ref{dfchowwt}, and Theorem \ref{thasfed} imply that, for some $c \in \mathbb{N}$ such that $cA$ has integral eigenvalues, we have
\begin{equation} \label{errtsd}
		\lim_{t \to + \infty} \frac{d}{dt} \mathscr{D}_m (H_t) = \frac{1}{c} \left(  \mathrm{Ding} ( \CX_{cA} , \CL_{cA} ) + \mathrm{Chow}_m ( \CX_{cA} , \CL_{cA} ) \right)
\end{equation}
where $( \CX_{cA} , \CL_{cA} )$ is a test configuration as defined in Proposition \ref{pprt}. The above quantity is nonnegative, and zero if and only if $( \CX_{cA} , \CL_{cA} )$ is product, by the $F$-polystability.

	For a general $H_0$-hermitian $A \in \mathfrak{gl} (H^0 (X , -mK_X))$, by passing to a basis that is $H_0$-unitarily equivalent to the reference basis if necessary, we write
	\begin{equation*}
		A = \mathrm{diag} (\lambda_1 , \dots , \lambda_{N_m})
	\end{equation*}
	with $\lambda_1 \ge \cdots \ge \lambda_{N_m}$. We then have
	\begin{equation*}
		\lim_{t \to + \infty} \frac{d}{dt} \mathscr{D}_m ( H_t) = \frac{2}{m} \sum_{i=1}^{N_m} \lambda_i \left( C_i (\CX_0) - \frac{1}{N_m} \right),
	\end{equation*}
	by Proposition \ref{lmasll}, where $\CX_0$ is the flat limit (\ref{dffllm}) defined by $A$. We now take an approximation $\{ A_p \}_{p=1}^{\infty} \subset \mathfrak{gl} (H^0 (X , -mK_X))$ of $A$ by $H_0$-hermitian matrices with rational eigenvalues, as in \cite[Lemma 3.15]{hk18}, so that the following hold: $A_p \to A$ as $p \to \infty$ (say in the Hilbert--Schmidt norm), and the flat limit (\ref{dffllm}) defined by $A_p$ is equal to $\CX_0$ (i.e.~the one defined by $A$) for all large enough $p$. Writing $H_{t,p} := \exp (-tA_p^*) \exp (-tA_p)$, and $\lambda_{1,p} , \dots , \lambda_{N_m , p}$ for the eigenvalues of $A_p$, again by Proposition \ref{lmasll} we find
	\begin{equation*}
		\lim_{t \to + \infty} \frac{d}{dt} \mathscr{D}_m ( H_{t,p}) = \frac{2}{m} \sum_{i=1}^{N_m} \lambda_{i,p} \left( C_i (\CX_0) - \frac{1}{N_m} \right),
	\end{equation*}
	where an important point is that the flat limit $\CX_0$ remains unchanged for $A_p$ when $p$ is large enough. Since $A_p \to A$ as $p \to \infty$, we get
	\begin{equation} \label{ernntsd}
		\lim_{t \to + \infty} \frac{d}{dt} \mathscr{D}_m (H_t)  = \lim_{p \to \infty} \lim_{t \to + \infty} \frac{d}{dt} \mathscr{D}_m (H_{t,p}) \ge 0.
	\end{equation}
	for any $H_0$-hermitian $A \in \mathfrak{gl} (H^0 (X , -mK_X))$.

	We prove that the inequality (\ref{ernntsd}) is strict unless $A \in \mathfrak{aut} (X)_r$. Suppose for contradiction that there exists $A \not\in \mathfrak{aut} (X)_r$, $H_0$-hermitian, such that
	\begin{equation} \label{eqzrctd}
		\lim_{t \to + \infty} \frac{d}{dt} \mathscr{D}_m ( H_{t}) = 0
	\end{equation}
	holds for the Bergman geodesic ray $\{ H_t \}_{t \ge 0}$ as defined in (\ref{eqbggda}). By applying the argument in \cite[proof of Lemma 3.17]{hk18}, we can write $A$ as a sum 
	\begin{equation} \label{eqalbtrti}
		A = A_{\alpha} + A_{\beta}
	\end{equation}
	of two $H_0$-hermitian matrices such that, in the diagonalising basis for $A$, we have
	\begin{enumerate}
	\renewcommand{\labelenumi}{(\roman{enumi})}
		\item $A_{\alpha} = \mathrm{diag} (\alpha_1 , \dots , \alpha_{N_m})$ with $\alpha_i \in \rtn$ for all $i$ and $\alpha_1 \ge \cdots \ge \alpha_{N_m}$,
		\item $A_{\beta} = \mathrm{diag} (\beta_1 , \dots , \beta_{N_m})$ with $\beta_i \in \rl$ for all $i$ and $\beta_1 \ge \cdots \ge \beta_{N_m}$,
		\item the flat limit (\ref{dffllm}) defined by $A$, $A_{\alpha}$, and $A_{\beta}$, are all equal.
	\end{enumerate}
	We define $H_{t,\alpha} := \exp (-tA_{\alpha}^*) \exp (-tA_{\alpha})$ and $H_{t,\beta} := \exp (-tA_{\beta}^*) \exp (-tA_{\beta})$. The properties above, together with Proposition \ref{lmasll}, imply that we have
	\begin{equation} \label{eqdcdmrir}
		\lim_{t \to + \infty} \frac{d}{dt} \mathscr{D}_m (H_{t}) = \lim_{t \to + \infty} \frac{d}{dt} \mathscr{D}_m (H_{t,\alpha}) + \lim_{t \to + \infty} \frac{d}{dt} \mathscr{D}_m (H_{t,\beta}) .
	\end{equation}
	Now (\ref{eqzrctd}), combined with (\ref{ernntsd}) and (\ref{eqdcdmrir}), necessarily implies
	\begin{equation*}
		\lim_{t \to + \infty} \frac{d}{dt} \mathscr{D}_m (H_{t,\alpha}) = \lim_{t \to + \infty} \frac{d}{dt} \mathscr{D}_m (H_{t,\beta}) = 0.
	\end{equation*}
	We again take $c \in \mathbb{N}$ so that $cA_{\alpha}$ has integral eigenvalues, and write $(\CX_{c A_{\alpha}} , \CL_{c A_{\alpha}})$ for the test configuration associated to $c A_{\alpha}$. We further note
	\begin{equation*}
		\lim_{t \to + \infty} \frac{d}{dt} \mathscr{D}_m (H_{t,\alpha}) = \frac{1}{c} \left( \mathrm{Ding} (\CX_{c A_{\alpha}} , \CL_{c A_{\alpha}}) + \mathrm{Chow}_m (\CX_{c A_{\alpha}} , \CL_{c A_{\alpha}}) \right) = 0,
	\end{equation*}
	which implies that $(\CX_{c A_{\alpha}} , \CL_{c A_{\alpha}})$ must be product by the $F$-polystability. Thus $\CX_{c A_{\alpha}}$ is isomorphic to $\iota (X) \times \cx \subset \prj (H^0 (X , -mK_X)^{\vee} ) \times \cx $ and its central fibre $\CX_0$, which is the flat limit (\ref{dffllm}) defined by $A_{\alpha}$, must be isomorphic to $\iota (X)$. Noting that $\CX_{c A_{\alpha}} \isom \iota (X) \times \cx$ must be invariant under the action of $e^{-tA_{\alpha}}$, we find $A_{\alpha} \in \mathfrak{aut}(X)$ by the equivariant embedding theorem (\ref{eqautffebd}) and hence $A_{\alpha} \in \mathfrak{aut} (X)_r$ since $A_{\alpha}$ is $H_0$-hermitian. Recall that $A_{\beta} \in \mathfrak{gl}(H^0 (X , -mK_X))$ was defined in such a way that its flat limit agrees with that of $A_{\alpha}$, which is $\CX_0 = \iota (X)$. Noting that $e^{-tA_{\beta}}$ acts on $\CX_0 = \iota (X)$ we find $A_{\beta} \in \mathfrak{aut} (X)_r$, again by (\ref{eqautffebd}), but this implies $A = A_{\alpha} + A_{\beta} \in \mathfrak{aut} (X)_r$ and contradicts our original hypothesis $A \not\in \mathfrak{aut} (X)_r$. Thus, we need to have
	\begin{equation*}
		\lim_{t \to + \infty} \frac{d}{dt} \mathscr{D}_m (H_{t}) > 0,
	\end{equation*}
	for all $H_0$-hermitian endomorphisms $A$ that are not contained in $\mathfrak{aut} (X)_r$.


We now prove that $\mathscr{D}_m$ is invariant under the $\mathrm{Aut}_0 (X)_r$-action. Recalling $\mathrm{Aut}_0 (X)_r = K_0^{\cx}$ in the notation of \S \ref{scqkrgs}, we first show that for any $H \in \CB_m$ and any $A \in \ai \mathrm{Lie} (K_0)$ we have $\frac{d}{dt} \mathscr{D}_m (H_t) = 0$ for the family $H_t := e^{t A} \cdot H$ defined by the action (\ref{eqautaobm}). Noting that $A$ is diagonalisable, if $A$ has integral eigenvalues we find that the family $\{ \mathrm{FS} (H_t) \}_{t \ge 0}$ (which may not be associated to a geodesic emanating from $H$ with respect to the bi-invariant metric, unless $A$ is hermitian with respect to $H$) admits $( \CX_A , \CL_A )$ as its non-Archimedean limit (cf.~Lemma \ref{lmbgnam}) by Lemma \ref{lmautaobm}, up to scaling if necessary, where $( \CX_A , \CL_A )$ is the product test configuration generated by $A$. We thus find, recalling Remark \ref{rmdgtrinv}, that
\begin{equation*}
		\lim_{t \to + \infty} \frac{d}{dt} \mathscr{D}_m (H_{t}) = \mathrm{Ding} (\CX_{A} , \CL_{A}) + \mathrm{Chow}_m (\CX_{A} , \CL_{A}) = 0
\end{equation*}
by the $F$-polystability; Theorem \ref{thasfed} is stated only for the Bergman geodesic rays but the result above holds since $\{ \mathrm{FS} (H_t) \}_{t \ge 0}$ is a psh ray (which is in fact a geodesic in $\CH$ induced by the holomorphic vector field $A$ by Lemma \ref{lmautaobm}) that admits $( \CX_A , \CL_A )$ as its non-Archimedean limit (see also the proof of Lemma \ref{rmfstgs} presented later). We now observe the following: for any $B \in \mathfrak{aut} (X)$ and the associated family $H_t = e^{t B} \cdot H$, writing $\{ Z'_i \}_{i=1}^{N_m}$ for an $H$-orthonormal basis and $\tilde{h}_H$ for the hermitian metric on the hyperplane bundle over $\prj^{N_m-1}$ with respect to $H$, we have
	\begin{align}
		\frac{d}{dt} \mathscr{D}_m (H_{t}) &= \frac{1}{m}  \sum_{i,j=1}^{N_m} (B_{ij} + B^{\dagger}_{ij}) \left( \fint_{e^{-tB} \cdot \iota (X)}  \frac{ \tilde{h}_H ( Z'_i ,  Z'_j)}{\sum_{l=1}^{N_m} \Vert  Z'_l \Vert_{\tilde{h}_H}^2} \tilde{h}_H^{1/m}  - \frac{1}{N_m} \delta_{ij} \right) \notag \\
		&= \frac{1}{m}  \sum_{i,j=1}^{N_m} (B_{ij} + B^{\dagger}_{ij}) \left( \fint_{\iota (X)}  \frac{ \tilde{h}_H (  Z'_i ,  Z'_j)}{\sum_{l=1}^{N_m} \Vert  Z'_l \Vert_{\tilde{h}_H}^2} \tilde{h}_H^{1/m}  - \frac{1}{N_m} \delta_{ij} \right), \label{eqdvdmautg}
	\end{align}
where all the matrices are represented with respect to $\{ Z'_i \}_{i=1}^{N_m}$ and $B^{\dagger}$ is the hermitian conjugate of $B$ with respect to $H$, by recalling the proof of Lemma \ref{lmdvffs}. We thus find that, for a general $A \in \ai \mathrm{Lie} (K_0)$ and the associated family $H_t = e^{t A} \cdot H$, we have
\begin{equation*}
	\frac{d}{dt} \mathscr{D}_m (H_{t}) = \lim_{s \to + \infty} \frac{d}{ds} \mathscr{D}_m (H_{s}) =  0
\end{equation*}
for all $t \ge 0$, by approximating $A$ by an $H_0$-hermitian sequence $\{ A_p \}_{p=1}^{\infty} \subset \ai \mathrm{Lie} (K_0)$ with rational eigenvalues and converging to $A$, by using \cite[Lemma 3.15]{hk18} as before, noting that the associated flat limits are all equal to $\iota (X)$. Note moreover that, for any $A \in \ai \mathrm{Lie} (K_0)$ and any $u \in K_0$, by replacing $H$ (resp.~$\{ Z'_i \}_{i=1}^{N_m}$) by $u \cdot H$ (resp.~$\{ u^{-1} Z'_i \}_{i=1}^{N_m}$) in the above we have
\begin{equation*}
	\frac{d}{dt} \mathscr{D}_m (e^{tA}u \cdot H) = 0,
\end{equation*}
again by the $F$-polystability, as $\{ e^{tA}u \cdot H \}_{t \ge 0}$ admits $( \CX_A , \CL_A )$ (rotated by $u$ in $\prj^{N_m-1}$) as its non-Archimedean limit if $A$ has integral eigenvalues. This proves the claimed $\mathrm{Aut}_0 (X)_r$-invariance of $\mathscr{D}_m$ by the global Cartan decomposition for $\mathrm{Aut}_0 (X)_r = K_0^{\cx}$, since we have $\frac{d}{dt} \mathscr{D}_m (e^{tA'} \cdot H) = 0$ for all $A' \in \mathrm{Lie}(K_0)$ and all $t \ge 0$ by applying (\ref{eqdvdmautg}) to $A'$ (otherwise it contradicts the compactness of $K_0$).

The conclusion of the argument so far is that $\mathscr{D}_m$ is a smooth function on $\CB_m$ that is convex along the Bergman geodesic rays emanating from $H_0$ and invariant under the action of $\mathrm{Aut}_0 (X)_r$, with a strictly positive asymptotic slope along a Bergman geodesic ray not contained in the $\mathrm{Aut}_0 (X)_r$-orbit (or equivalently the $\mathrm{Aut}_0 (X)$-orbit, as long as we consider the Bergman geodesic rays emanating from $H_0$ that are by definition generated by the $H_0$-hermitian endomorphisms). Since $\CB_m$ is a complete Riemannian manifold with respect to the bi-invariant metric (whose geodesics are precisely the Bergman geodesics) on which $\mathrm{Aut}_0 (X)_r$ acts isometrically by Lemma \ref{lmautaobm}, this implies that $\CB_m$ admits a critical point which is unique modulo the action of $\mathrm{Aut}_0 (X)_r$. In fact the critical point of $\mathscr{D}_m$ is unique modulo the action of $\mathrm{Aut}_0 (X)$; this can be seen by considering (\ref{eqdvdmautg}) for a general $B \in \mathfrak{aut} (X)$ and the family $H_t = e^{tB} \cdot H$ emanating from the critical point $H$ of $\mathscr{D}_m$, for which the term inside the bracket is zero by Proposition \ref{ppacbal}.

	For the proof of the reverse direction, suppose that $(X,-K_X)$ admits an anticanonically balanced metric $H$ at level $m$ which is unique up to the $\mathrm{Aut}_0 (X)$-action. This in particular implies that $\mathscr{D}_m$ must be bounded below over $\CB_m$ by convexity (Theorem \ref{thberndt}), which in turn implies that $\mathscr{D}_m$ must be invariant under the group action $\mathrm{Aut}_0 (X) \actson \CB_m$, as otherwise there would exist an $\mathrm{Aut}_0 (X)$-orbit along which $\mathscr{D}_m$ tends to $- \infty$, as we saw above in (\ref{eqdvdmautg}). Since $\mathscr{D}_m$ is strictly convex along the Bergman geodesic rays not contained in the $\mathrm{Aut}_0 (X)$-orbit by Theorem \ref{thberndt}, we find that for any $H'_0 \in \CB_m$ and any Bergman geodesic ray $\{ H'_t \}_{t \ge 0}$ emanating from $H'_0$ we have
	\begin{equation} \label{eqmtipst}
		\lim_{t \to + \infty} \frac{d}{dt} \mathscr{D}_m (H'_t) \ge 0 ,
	\end{equation}
	with equality if and only if $\{ H'_t \}_{t \ge 0}$ is contained in the $\mathrm{Aut}_0 (X)$-orbit of $H'_0$; this is a consequence of $\mathscr{D}_m$ being proper modulo the $\mathrm{Aut}_0 (X)$-action over $\CB_m$, which in turn follows from the existence of the unique global minimum of $\mathscr{D}_m$ over $\CB_m$ up to $\mathrm{Aut}_0 (X)$. Suppose that we have a very ample test configuration $(\CX_A , \CL_A)$ of exponent $m$, where $A$ is the generator of the one-parameter subgroup as in Proposition \ref{pprt}. We take a basis $\{ s'_i \}_{i=1}^{N_m}$ for $H^0 (X , -mK_X)$ such that $A$ is diagonal (with integral eigenvalues), and define $H'_0$ to be the positive definite hermitian form such that $\{ s'_i \}_{i=1}^{N_m}$ is $H'_0$-orthonormal. Arguing as in (\ref{errtsd}), we find (by Theorem \ref{thasfed})
	\begin{equation*}
		\lim_{t \to + \infty} \frac{d}{dt} \mathscr{D}_m (H'_t) = \mathrm{Ding} (\CX_A , \CL_A) + \mathrm{Chow}_m (\CX_A , \CL_A),
	\end{equation*}
	which is nonnegative by (\ref{eqmtipst}), with equality if and only if $A \in \mathfrak{aut} (X)$, i.e.~$(\CX_A , \CL_A)$ is product. This establishes the $F$-polystability as required.
\end{proof}

We now note that the uniqueness of anticanonically balanced metrics modulo automorphism, if they exist, can also be proved directly from the second statement of Theorem \ref{thberndt} by Berndtsson.

\begin{proposition} \label{ppuqqdfn} 
	If $H, H' \in \CB_m$ are both critical points of $\mathscr{D}_m$, then there exists $a \in \mathrm{Aut}_0 (X)$ such that $H' = a \cdot H$.
\end{proposition}


We omit the proof, which follows from Theorem \ref{thberndt}, Lemma \ref{lmautaobm}, and \cite[Theorem 1.1]{Lempert2021}; the details are also written in \cite[Proof of Proposition 3.3]{tak19}. We are now ready to prove Corollary \ref{mthstc}.

\begin{proof}[Proof of Corollary \ref{mthstc}]
Since $\mathrm{Aut}_0 (X)$ is assumed to be trivial, a test configuration is product if and only if it is trivial. Then $\delta_m >1$ easily follows from the $F$-stability by Theorem \ref{mthst} and Rubinstein--Tian--Zhang \cite[Theorem 2.3 (ii)]{rtz}. Conversely, if $\delta_m >1$ then there exists an anticanonically balanced metric at level $m$ by Rubinstein--Tian--Zhang \cite[Theorem 2.3 (i)]{rtz}, which is unique by Proposition \ref{ppuqqdfn} since $\mathrm{Aut}_0 (X)$ is trivial. Thus the claimed $F$-stability follows from Theorem \ref{mthst}.
\end{proof}

We finally comment on the following fact, which is clearly well-known to the experts but we provide a proof as it will be referred to later.

\begin{lemma} \label{rmfstgs}
	If $(X, -K_X)$ is $F$-stable at level $m$ then $\mathrm{Aut}_0 (X)_r$ is trivial.
\end{lemma}

\begin{proof}
Suppose $\mathfrak{aut} (X)_r \neq 0$. By taking a rational approximation as in the proof of Theorem \ref{mthst}, there exists a nonzero $A \in \ai \mathrm{Lie} (K_0)$ with integral eigenvalues (i.e.~the generator of a one-parameter subgroup $\cx^* \to \mathrm{Aut}_0 (X)_r \subset GL(H^0 (X , -mK_X))$ in the category of algebraic groups), which gives rise to a product test configuration $(\CX , \CL)$. We then observe $\mathrm{Ding} (\CX , \CL) = -  \mathrm{Ding} (\CX' , \CL')$ and $\mathrm{Chow}_m (\CX , \CL) = - \mathrm{Chow}_m (\CX' , \CL')$ for the product test configuration $(\CX' , \CL')$ generated by $-A \in \mathfrak{aut} (X)_r$; the first equality holds since in this case the Ding invariant, which is the asymptotic slope of $\mathscr{D}$ along a holomorphic vector field as we saw in the proof of Theorem \ref{mthst}, is equal to the classical Futaki invariant (see the original formula in \cite{Futaki83}, or \cite[(3.4)]{Berman16} and \cite[Proposition 2.2.2]{dontoric}), and the second equality follows since the Chow weight is the (genuine, finite dimensional) GIT weight of $A$ acting on the Chow point represented by $\iota (X) \subset \prj (H^0 (X , -mK_X)^{\vee})$ (see e.g.~\cite{rt07,Futaki12} for more details). Thus it is impossible for all the invariants above to be strictly positive.
\end{proof}

	
Note that the above proof is also related to the vanishing of the higher Futaki invariants \cite{Futaki04}, as discussed in e.g.~\cite[Proposition 5.4]{st19}.


\subsection{Extension to K\"ahler--Ricci $g$-solitons} \label{scekrgs}

We now prove a similar result for the K\"ahler--Ricci $g$-soliton case. First note that we have, from Theorem \ref{thasfed} and (\ref{eqslegm}),
\begin{equation*}
	\lim_{t \to + \infty} \frac{d}{dt} \mathscr{D}^g_m (H_t) = \mathscr{D}^{g , \mathrm{NA}}_m (\CX_A , \CL_A)
\end{equation*}
for a Bergman geodesic ray $\{ H_t \}_{t \ge 0} \subset \CB^T_m$ emanating from $H_0$, generated by an $H_0$-hermitian matrix $A \in \mathfrak{gl} (H^0 (X , -mK_X) )$ with integral eigenvalues and commuting with the $T$-action. Given this formula, it is straightforward to adapt the proof of Theorem \ref{mthst} to the $g$-soliton case.

\begin{theorem} \label{thkrsbm}
A Fano manifold $(X , -K_X)$ admits a $T$-invariant anticanonically $g$-balanced metric at level $m$, which is unique up to $\mathrm{Aut}_0 (X , T^{\cx})$, if and only if it is $F$-polystable in the sense of $g$-solitons at level $m$.
\end{theorem}

\begin{proof}
	The proof is almost the same as that of Theorem \ref{mthst}, since the difference $\mathscr{D}^g_m (H_t) - \mathscr{D}_m (H_t)$ is linear in $A$; we may even assume that $\mathscr{E}^g_m (H_t)$ is constantly equal to zero, by using the translation invariance of the functional $\mathscr{D}^g_m$ (cf.~Remark \ref{rmdgtrinv}). The only subtlety is that the domain of $\mathscr{D}^g_m$ is $\CB_m^T$ rather than $\CB_m$, but the argument works word by word since we only need to consider the matrices that commute with (i.e.~preserve the weight decomposition of) $T^{\cx}$, and we can construct the rational approximation that also commutes with the $T^{\cx}$-action (recall also Lemma \ref{lmtccmtrs}); this can be done by noting that $A$ can be diagonalised preserving the weight decomposition (\ref{eqwtdct}), since $A$ is assumed to commute with the $T^{\cx}$-action, and so we can repeat the previous argument which explicitly deals with the eigenvalues. We finally recall that $\CB^T_m$ is a Riemannian symmetric space with respect to the bi-invariant metric, as pointed out in \S \ref{scqkrgs}.
\end{proof}

\begin{remark} \label{rmfstgss}
	The fact mentioned in Lemma \ref{rmfstgs} applies to the $g$-solitons, by observing
	\begin{align*}
		\mathscr{D}^{g , \mathrm{NA}}_m (\CX , \CL) =& \mathrm{Ding} (\CX , \CL) + \mathrm{Chow}_m (\CX , \CL) \\
		& + \frac{\mathrm{tr} (A+A^*)}{m h^0 (X , -mK_X)} - \frac{1}{m N_m \overline{g_m}} \sum_{\lambda \in P_m} g ( \lambda / m) \mathrm{tr} \left( (A+A^*)|_{R_{m , \lambda}} \right) .
	\end{align*}
	and arguing exactly the same as before, by noting that the last two terms clearly change the sign when $A$ is replaced by $-A$. In particular, if $(X ,-K_X)$ is $F$-stable in the sense of $g$-solitons at level $m$ then $\mathrm{Aut}_0 (X,T^{\cx})_r$ must be trivial; an important difference to Lemma \ref{rmfstgs} is that $\mathrm{Aut}_0 (X,T^{\cx})_r$ can never be trivial if $T^{\cx}$ is nontrivial. Thus $(X , -K_X)$ cannot be $F$-stable in the sense of $g$-solitons if $T^{\cx}$ is nontrivial.
\end{remark}

Note that the above balanced metric is never unique if $T^{\cx}$ is nontrivial, and unique only up to automorphisms that commute with $T^{\cx}$. This is a well-documented phenomenon and follows also from Remark \ref{rmfstgss}; for there to be a balanced metric $\mathscr{D}^g_m$ must be bounded below by convexity, but this forces $\mathscr{D}^g_m$ to be invariant under $\mathrm{Aut}_0 (X , T^{\cx})_r$ (the proof is exactly the same as the one for Theorem \ref{mthst}). Note furthermore that the analogue of Proposition \ref{ppuqqdfn} holds: if $H, H' \in \CB_m^T$ are both critical points of $\mathscr{D}^g_m$, then there exists $a \in \mathrm{Aut}_0 (X , T^{\cx})$ such that $H' = a \cdot H$. Indeed, we connect $H_0 := H$ and $H_1 := H'$ by a Bergman geodesic path $\{ H_t \}_{0 \le t \le 1}$ in $\CB^T_m$, and write $H_t = e^{B^{\dagger} t} H e^{Bt}$ for some $B \in \mathfrak{gl} (H^0 (X , -mK_X))$ that is hermitian with respect to $H$ and commutes with $\mathrm{Lie} (T)$. We find $B \in \mathfrak{aut}(X)$ exactly as in the proof of Proposition \ref{ppuqqdfn}, and hence $B \in \mathfrak{aut}(X , T^{\cx})$ as it commutes with $\mathrm{Lie} (T)$.

With the above remark in mind we get the following corollary, by combining the above theorem with \cite[Theorem 6.13]{rtz}; see \S \ref{aoidmgs} for the definition of $\delta^g_m$.

\begin{corollary} \label{crstdm}
	If $(X,-K_X)$ is $F$-polystable in the sense of $g$-solitons then $\delta^g_m \ge 1$. If $\delta^g_m > 1$ then $(X,-K_X)$ is $F$-polystable in the sense of $g$-solitons at level $m$.
\end{corollary}

Note that Rubinstein--Tian--Zhang proved $\delta^g_m =1$ for $g \equiv 1$ and $X = \prj^n$ \cite[Corollary 7.2]{rtz}.

\subsection{Extension to coupled K\"ahler--Einstein metrics} \label{sceckem}

We now prove the analogue of Theorem \ref{mthst} for the coupled \ke case. We first prove the following result, which can be considered as an analogue of the argument that appears in \cite[\S 5]{HWN19}, in which we do not necessarily assume $\CY = \CX_1 = \cdots = \CX_k$ or $\CL_{\CY} = \sum_{i=1}^k \CL_i$ (see Remark \ref{rmctcod}).


\begin{theorem} \label{ppcdgas}
Let $\{ (H_{1,t} , \dots , H_{k,t}) \}_{t \ge 0} \subset \bm{\CB}_m$ be a $k$-tuple of Bergman geodesic rays generated by $(A_1 , \dots , A_k) \in \mathfrak{gl} (H^0(X , mL_1) \times \cdots \times \mathfrak{gl} (H^0(X , mL_k)$, which defines a $k$-tuple of psh rays $(\mathrm{FS}_1 (H_{1,t}) , \dots , \mathrm{FS}_k (H_{k,t}))$ in $\bm{\CH}$. Suppose that $A_i$ is hermitian, with integral eigenvalues, with respect to the reference hermitian form $H_{i,0}$ for each $i=1 , \dots , k$, corresponding to a $k$-tuple of very ample test configurations $(\CX_{A_1} , \CL_{A_1}) , \dots , (\CX_{A_k} , \CL_{A_k})$ of exponent $m$. Then we have
	\begin{equation*}
		\lim_{t \to + \infty} \frac{\mathscr{D}^{\mathrm{coupled}} (\mathrm{FS}_1 (H_{1,t}) , \dots , \mathrm{FS}_k (H_{k,t}))}{t} = \mathrm{Ding} \left( (\CX_{A_i} , \CL_{A_i})_{i=1}^k \right).
	\end{equation*}
\end{theorem}

\begin{proof}
	First recall that we have $H_{i,t}= e^{-A^*_i t} e^{-A_i t}$ as in (\ref{eqbggda}), and we define
	\begin{equation*}
		\bm{H}_t := H_{1,t} \otimes \cdots \otimes H_{k,t}.
	\end{equation*}
	Writing $\{ \bm{s}_{\bm{j}} \}_{\bm{j}}$ for the reference basis for $H^0 (X , mL_1) \otimes \cdots \otimes H^0 (X , mL_k)$, as in (\ref{eqcprbsb}), and defining
	\begin{equation*}
	\bm{g}_t := e^{A_1 t} \otimes \cdots \otimes e^{A_k t} ,
	\end{equation*}
	we find that $\{ \bm{g}_t \cdot \bm{s}_{\bm{j}} \}_{\bm{j}}$ is an $\bm{H}_t$-orthonormal basis, where the action $\cdot$ is the natural tensor product action on $H^0 (X , mL_1) \otimes \cdots \otimes H^0 (X , mL_k)$. After an appropriate scaling if necessary, we find by Lemma \ref{lmbgnam} that the Bergman geodesic ray $\{ \mathrm{FS} (\bm{H}_t) \}_{t \ge 0} \subset \CH$ admits a non-Archimedean metric $\phi^{\mathrm{NA}}$ on $-K_X$ as non-Archimedean limit, where $\phi^{\mathrm{NA}}$ is represented by the very ample test configuration defined as the Zariski closure of $\iota_{\mathrm{coupled}} (X) \subset \prj( H^0 (X , mL_1)^{\vee} \otimes \cdots \otimes H^0 (X , mL_k)^{\vee} )$ under the one-parameter subgroup $\bm{g}^{\vee}_t = \tau^{A_1} \otimes \cdots \otimes \tau^{A_k}$, by recalling the identification $\tau = e^{-t}$ (and the sign change by taking the dual) as in (\ref{eqcvtaut}). The resulting test configuration is exactly the one that is generated by the $\cx^*$-actions of $(\CX_{A_i} , \CL_{A_i})_{i=1}^k$, which is denoted by $(\CY , \CL_{\CY})$.
	
	Note moreover that Lemma \ref{lmrfmc} gives $\mathscr{L}^{\mathrm{coupled}} (\mathrm{FS}_1 (H_{1,t}) , \dots , \mathrm{FS}_k (H_{k,t})) = \mathscr{L} (\mathrm{FS} ( \bm{H}_t ))$.
	Thus we find, by Theorem \ref{thasfed}, that
	\begin{equation*}
		\lim_{t \to + \infty} \frac{\mathscr{L}^{\mathrm{coupled}} (\mathrm{FS}_1 (H_{1,t}) , \dots , \mathrm{FS}_k (H_{k,t}) )}{t} = \lim_{t \to + \infty} \frac{\mathscr{L} (\mathrm{FS} ( \bm{H}_t ))}{t} = - 1 + \mathrm{lct} (\CY^{\nu} , D_{\CY^{\nu} , \CL_{\CY^{\nu}}} ; \CY^{\nu}_0).
	\end{equation*}
	Combined with the asymptotic slopes for $\mathscr{E} (\mathrm{FS}_i (H_{i,t}))$, again by Theorem \ref{thasfed}, we get the result.
\end{proof}

The above theorem immediately implies the following.

\begin{lemma}
Let $\{ (H_{1,t} , \dots , H_{k,t}) \}_{t \ge 0} \subset \bm{\CB}_m$ be a $k$-tuple of Bergman geodesic rays defined exactly as in Theorem \ref{ppcdgas}. Then
	\begin{equation*}
		\lim_{t \to + \infty} \frac{d}{dt} \mathscr{D}^{\mathrm{coupled}}_m (H_{1,t} , \dots , H_{k,t}) = \mathrm{Ding} \left( (\CX_{A_i} , \CL_{A_i})_{i=1}^k \right) + \sum_{i=1}^k \mathrm{Chow}_{m} (\CX_{A_i} , \CL_{A_i}).
	\end{equation*}
\end{lemma}

\begin{proof}
The claim is obvious by noting, as in \cite{st19}, that $\mathscr{D}^{\mathrm{coupled}}_m (H_{1,t} , \dots , H_{k,t})$ can be written as
\begin{equation*}
\mathscr{D}^{\mathrm{coupled}} (\mathrm{FS}_1 (H_{1,t}) , \dots , \mathrm{FS}_k (H_{k,t})) + \sum_{i=1}^k \left( \mathscr{E} (\mathrm{FS}_i (H_{i,t})) - \mathscr{E}_{i,m} (H_i) \right),
\end{equation*}
by recalling Lemma \ref{lmrfmc}, and applying Lemma \ref{lmdvffs}, Definition \ref{dfchowwt}, Theorems \ref{thasfed} and \ref{ppcdgas}.
\end{proof}

We now prove the following coupled version of Theorem \ref{mthst}.

\begin{theorem} \label{thckebm}
	$(X , -K_X; L_1 , \dots , L_k)$ admits a coupled anticanonically balanced metric at level $m$, which is unique up to $\mathrm{Aut}_0 (X)$, if and only if
	\begin{equation*}
		\mathrm{Ding} \left( (\CX_i , \CL_i)_{i=1}^k \right) + \sum_{i=1}^k \mathrm{Chow}_{m} (\CX_i , \CL_i) \ge 0
	\end{equation*}
	holds for any $k$-tuple of very ample test configurations $ (\CX_1 , \CL_1) , \dots , (\CX_k , \CL_k )$ for $(X, L_1) , \dots , (X , L_k)$ respectively, each of exponent $m$, with equality if and only if the test configuration $(\CY , \CL_{\CY})$ generated by their $\cx^*$-actions is product.
\end{theorem}

\begin{proof}
	The proof in \S \ref{scpfmtmc} applies almost word by word: for each $(A_1 , \dots , A_k) \in \mathfrak{gl} (H^0(X , mL_1) \times \cdots \times \mathfrak{gl} (H^0(X , mL_k)$ we choose an $H_{i,0}$-orthonormal basis for $H^0 (X, mL_i)$ such that $A_i$ is diagonal, for $i=1 , \dots , k$, and argue as before. The only nontrivial point, however, is the following: when we choose a rational approximation $\{ A_{i,p} \}_{p=1}^{\infty}$ of $A_i$ such that the associated flat limit defined by $A_{i,p}$ is the same as the one defined by $A_i$ for all large enough $p$, the flat limit $\CY_0 \subset \prj (H^0 (X , mL_1)^{\vee} \otimes \cdots \otimes H^0 (X , mL_k)^{\vee} )$ defined by $e^{- A_{1} t} \otimes \cdots \otimes e^{- A_{k}t}$ is also the same as the one defined by $e^{- A_{1,p} t} \otimes \cdots \otimes e^{- A_{k,p}t}$ for all large enough $p$. This is indeed true, since the weight of $(A_1 , \dots , A_k)$ on $H^0 (X , mL_1)^{\vee} \otimes \cdots \otimes H^0 (X , mL_k)^{\vee} $, with $A_i = \mathrm{diag} (\lambda_{i,1} , \dots , \lambda_{i, N_{i,m}} )$, is given by the sum (and hence a $\itg$-linear combination) of these weights as $\sum_{i=1}^k \lambda_{i, j_i}$ ($j_i \in \{ 1 , \dots , N_{i,m} \}$ and $i=1 ,\dots , k $); thus the proof of \cite[Lemmas 3.15 and 3.17]{hk18}, which was used in the proof of Theorem \ref{mthst}, applies word by word. Likewise, when we choose $A_{i, \alpha}$ and $A_{i , \beta}$ as (\ref{eqalbtrti}) in the proof of Theorem \ref{mthst}, we can show that the flat limits defined by $e^{- A_{1} t} \otimes \cdots \otimes e^{- A_{k}t}$, $e^{- A_{1,\alpha} t} \otimes \cdots \otimes e^{- A_{k,\alpha}t}$, and $e^{- A_{1,\beta} t} \otimes \cdots \otimes e^{- A_{k,\beta}t}$ are all equal. This ensures that the proof given in \S \ref{scpfmtmc} applies word by word.
\end{proof}

We also get the following coupled version of Corollary \ref{mthstc}; see \S \ref{aoidmcpd} for the definition of $\delta^{\textup{coupled}}_{m}$.

\begin{corollary} \label{crcpdm}
	Suppose that $\mathrm{Aut}_0 (X)$ is trivial. Then $\delta^{\textup{coupled}}_{m} > 1$ if and only if
	\begin{equation*}
		\mathrm{Ding} \left( (\CX_i , \CL_i)_{i=1}^k \right) + \sum_{i=1}^k \mathrm{Chow}_{m} (\CX_i , \CL_i) > 0
	\end{equation*}
	holds for any $k$-tuple of very ample test configurations $ (\CX_1 , \CL_1) , \dots , (\CX_k , \CL_k )$ for $(X, L_1) , \dots , (X , L_k)$ respectively, each of exponent $m$ and none of which are trivial.
\end{corollary}

\begin{proof}
	Exactly the same as that of Corollary \ref{mthstc}, by using \cite[Theorem A.12]{rtz} and \cite[Proposition 3.3]{tak19}; recall from \S \ref{sccdinv} that $ (\CX_1 , \CL_1) , \dots , (\CX_k , \CL_k )$ are all nontrivial if and only if the test configuration $(\CY , \CL_{\CY})$ generated by their $\cx^*$-actions is.
\end{proof}

\appendix

\section{The $\delta_m$-invariant of Fujita--Odaka and its extensions} \label{appdmivfo}

\subsection{Original invariant $\delta_m$} \label{aoidm}

We recall the invariants $\delta_m$ and $\delta$, following \cite{bj20,fo18}. We start by quickly recalling the log canonical thresholds, following the exposition in \cite[\S 2.3]{Fujita18}.


\begin{definition} \label{dflct}
	Let $Y$ be a normal variety and $\Delta$ be an $\rl$-divisor on $Y$ such that $K_Y + \Delta$ is $\rl$-Cartier, where $\Delta$ is not necessarily effective. The pair $(Y,\Delta )$ is said to be \textbf{sub log canonical} if $a(E,Y,\Delta ) \ge -1$ holds for any proper birational morphism $\sigma : \tilde{Y} \to Y$ with $\tilde{Y}$ normal and any prime divisor $E$ on $\tilde{Y}$, where $a(E,Y,\Delta ) := \mathrm{ord}_E (K_{\tilde{Y}} - \sigma^* (K_Y +\Delta ))$. We say that $(Y,\Delta )$ is \textbf{log canonical} if it is sub log canonical and $\Delta$ is effective.
	
	For the pair $(Y,\Delta )$ as above, and for a nonzero effective $\rl$-Cartier divisor $B$ on $Y$, the \textbf{log canonical threshold of $B$ with respect to $(Y,\Delta )$} is defined as
	\begin{equation*}
		\mathrm{lct} (Y,\Delta ;B) := \sup \{ c \in \rl \mid (Y,\Delta +cB) \text{ is sub log canonical}. \};
	\end{equation*}
	if $(Y,\Delta +cB)$ is not sub log canonical for any $c \in \rl$, we set $\mathrm{lct} (Y,\Delta ;B) = - \infty$.
	
	If $Y$ is further assumed to be $\rtn$-Fano, i.e.~$-K_Y$ is an ample $\rtn$-Cartier divisor and $Y$ has at worst log terminal singularities, and $D$ is an effective $\rtn$-Cartier divisor on $Y$, we further define the \textbf{log canonical threshold of $Y$ along $D$} as
\begin{equation*}
	\mathrm{lct} (Y;D) := \sup \{ c \in \rl_{\ge 0} \mid (Y,cD) \text{ is log canonical}. \}.
\end{equation*}
\end{definition}

While it is good to have the above general definitions, in this appendix we apply them exclusively to the case when the variety in question is a smooth Fano variety. As in the body of the text, $X$ stands for a Fano manifold of complex dimension $n$.

\begin{definition}
	Let $X$ be a Fano manifold and $m$ be a positive integer, and define $N_m := \dim_{\cx} H^0 (X , - mK_X)$. We say that a divisor $D$ on $X$ is of \textbf{$m$-basis type} if there exists a basis $\{ s_i \}_{i=1}^{N_m}$ for $H^0 (X , -mK_X)$ such that
	\begin{equation*}
		D = \frac{1}{m N_m} \sum_{i=1}^{N_m} \mathrm{div} (s_i).
	\end{equation*}
\end{definition}

\begin{definition}
	Let $X$ be a Fano manifold and $m$ be a positive integer. The \textbf{$\delta_m$-invariant} of Fujita--Odaka is defined by
	\begin{equation*}
		\delta_m := \inf_D \mathrm{lct} (X;D)
	\end{equation*}
	where the infimum is taken over all divisors on $X$ that are of $m$-basis type. The \textbf{$\delta$-invariant} is defined by
	\begin{equation*}
		\delta := \limsup_{m \to \infty} \delta_m .
	\end{equation*}
\end{definition}


Blum--Jonsson \cite{bj20} provided a valuative formulation of the above invariants, which we now recall. 

\begin{definition}
	$F$ is said to be a \textbf{prime divisor over $X$} if there exists a normal variety $Y$ and a proper birational morphism $\sigma : Y \to X$ such that $F$ is a reduced irreducible Weil divisor on $Y$.
\end{definition}

Writing 
\begin{equation*}
	A_X (F) := 1+ \mathrm{ord}_F (K_{Y/X})
\end{equation*}
for the log discrepancy of $X$ along a prime divisor $F$ over $X$, where $K_{Y/X} = K_Y - \sigma^* K_X$ is the relative canonical bundle, it is well-known (see e.g.~\cite[\S 1.8]{bj20} and \cite[\S 1.5]{bhj1}) that we can write
\begin{equation*}
	\mathrm{lct} (X;D) = \inf_F \frac{A_X (F)}{\mathrm{ord}_F(D)},
\end{equation*}
where the infimum is taken over all prime divisors over $X$ with $ \mathrm{ord}_F(D) > 0$.

For a prime divisor $F$ over $X$, with $F \subset Y$ and a birational morphism $\sigma : Y \to X$, we write
\begin{equation*}
	S_m (F) := \frac{1}{mN_m} \sum_{j=1}^{\infty} \dim_{\cx} H^0 (Y,  \sigma^*(-mK_X) -jF ),
\end{equation*}
where we observe that all but finitely many summands are zero. A formula in \cite[proof of Lemma 2.2]{fo18} shows
\begin{equation*}
	S_m (F) = \sup_D \mathrm{ord}_F(D)
\end{equation*}
where the supremum is taken over all $m$-basis type divisors; in fact, \cite[proof of Lemma 2.2]{fo18} also shows that the supremum is attained by an $m$-basis type divisor $D_F$ which is compatible with the filtration of $H^0 (X , -mK_X)$ defined by $F$. Moreover \cite[proof of Theorem 2.1]{fo18} shows
\begin{equation*}
	\lim_{m \to \infty} S_m (F) = \frac{1}{\mathrm{Vol} (-K_X)} \int_0^{+ \infty} \mathrm{Vol} (-K_X -xF) dx =: S(F),
\end{equation*}
where $\mathrm{Vol} (-K_X -xF) := n! \limsup_{m \to \infty} \dim_{\cx} H^0 (Y,  \sigma^*(-mK_X) -xF ) / m^n$. Blum--Jonsson \cite[Proposition 4.3]{bj20} proved that 
\begin{equation*}
	\delta_m = \inf_F \frac{A_X (F)}{S_m (F)} ,
\end{equation*}
where the infimum is taken over all prime divisors over $X$. They moreover proved \cite[Theorem 4.4]{bj20} that the limit $\lim_{m \to \infty} \delta_m$ exists, and that
\begin{equation*}
	\delta = \lim_{m \to \infty} \delta_m =  \inf_F \frac{A_X (F)}{S (F)} ,
\end{equation*}
where the infimum is again taken over all prime divisors over $X$.

\subsection{K\"ahler--Ricci $g$-solitons version $\delta_m^g$}  \label{aoidmgs}

The invariants for the K\"ahler--Ricci $g$-solitons were defined by Rubinstein--Tian--Zhang \cite[Definition 6.4]{rtz}.

\begin{definition}
	Suppose that an algebraic torus $T^{\cx}$ acts on $X$, which induces a weight decomposition
	\begin{equation*}
	H^0 (X , -mK_X ) = \bigoplus_{\lambda \in P_m} R_{m ,\lambda}
	\end{equation*}
	as in \S \ref{scqkrgs}, where $P_m$ is the character lattice and $T^{\cx}$ acts by the character $\lambda$ on $R_{m ,\lambda}$. We say that $D$ is an \textbf{$(m,g)$-basis divisor} if there exists a basis $\{ s_{\alpha}^{(\lambda )} \}_{\alpha =1}^{N_{m , \lambda}}$ for each $R_{m ,\lambda}$ such that
	\begin{equation*}
		D = \frac{1}{m N_m \overline{g_m}} \sum_{\lambda \in P_m} \sum_{\alpha = 1}^{N_{m , \lambda}} g (\lambda / m) \{ s_{\alpha}^{(\lambda)} = 0 \}.
	\end{equation*}
	The invariant $\delta_m^g$ is defined by
	\begin{equation*}
		\delta_m^g := \inf_D \mathrm{lct} (X;D)
	\end{equation*}
	where the infimum is over all $(m,g)$-basis divisors. The invariant $\delta^g$ is defined by
	\begin{equation*}
		\delta^g := \limsup_{m \to \infty} \delta^g_m.
	\end{equation*}
\end{definition}
Analogues of the results in \S \ref{aoidm} are also given in \cite{rtz}. They define
\begin{equation*}
	S^g_m (F) := \frac{1}{m N_m \overline{g_m}} \sum_{\lambda \in P_m} \sum_{a = 1}^{\infty} g (\lambda / m) \dim_{\cx} \CF^{\ge a}_{\mathrm{ord}_F} R_{m, \lambda}
\end{equation*}
for a $T^{\cx}$-invariant prime divisor over $X$, where $\CF^{\ge a}_{\mathrm{ord}_F} R_{m, \lambda}$ is the subspace of $R_{m , \lambda}$ consisting of sections that vanish along $F$ with order at least $a$. We then have, by \cite[Lemma 6.5]{rtz},
\begin{equation*}
	\delta^g_m = \inf_F \frac{A_X (F)}{S^g_m (F)}
\end{equation*}
where the infimum is over all $T^{\cx}$-invariant prime divisors over $X$.

Setting
\begin{equation*}
\mathrm{Vol}^g (-K_X - xF) := \limsup_{m \to \infty} \frac{n!}{m^n} \sum_{\lambda \in P_m} g (\lambda / m) \dim_{\cx} \CF^{\ge x}_{\mathrm{ord}_F} R_{m, \lambda}
\end{equation*}
and defining
\begin{equation*}
	S^g (F) := \int_0^{+ \infty} \mathrm{Vol}^g (-K_X - xF) dx,
\end{equation*}
we have, by \cite[Proposition 6.14]{rtz},
\begin{equation*}
	\delta^g = \lim_{m \to \infty} \delta^g_m = \inf_F \frac{A_X (F)}{S^g (F)}
\end{equation*}
where the infimum is over all $T^{\cx}$-invariant prime divisors over $X$.

K.~Zhang \cite[Remark 5.3]{Zhang21} recently proved that $\delta^g >1$ implies the existence of K\"ahler--Ricci $g$-solitons.

\subsection{Coupled K\"ahler--Einstein version $\delta^{\mathrm{coupled}}_m$} \label{aoidmcpd}
The coupled version of the $\delta_m$-invariant was defined by Rubinstein--Tian--Zhang \cite[Definition A.2]{rtz} as follows.

\begin{definition}
	The \textbf{coupled $\delta_m$-invariant} is defined by
	\begin{equation*}
		\delta^{\text{coupled}}_{m} := \inf_{D_1 , \dots , D_k} \mathrm{lct} \left( X ; \sum_{i=1}^k D_i \right) ,
\end{equation*}
where the infimum is over a $k$-tuple of divisors $ (D_1 , \dots , D_k)$ such that $D_i \sim_{\rtn} L_i$ is an $m$-basis divisor for each $i=1 , \dots , k$. The \textbf{coupled $\delta$-invariant} is defined by
\begin{equation*}
	\delta^{\text{coupled}} := \limsup_{m \to \infty} \delta^{\text{coupled}}_m .
\end{equation*}
\end{definition}

Rubinstein--Tian--Zhang \cite{rtz} considers a more general definition in which $D_i$'s are allowed to have different exponents, i.e.~$D_i \sim_{\rtn} L_i$ being an $m_i$-basis divisor for possibly distinct integers $m_1 , \dots , m_k$, but it is not necessary for the results in this paper.

The valuative formula for $\delta_m$ and $\delta$ naturally generalises to the coupled case. For a prime divisor $F$ over $X$, with a proper birational morphism $\sigma : Y \to X$, we define
\begin{equation*}
	S_{m} (L_i;F) := \frac{1}{m h^0 (X , m L_i)} \sum_{j=1}^{\infty} \dim_{\cx} H^0 (Y,  \sigma^*(-m L_i) -jF ).
\end{equation*}
We have \cite[Lemma A.3]{rtz}
\begin{equation*}
	\delta^{\text{coupled}}_{m} = \inf_F \frac{A_X (F)}{\sum_{i=1}^k S_{m} (L_i; F)}
\end{equation*}
where the infimum runs over all prime divisors over $X$; the proof is exactly the same as above, by noting
\begin{equation*}
	\mathrm{lct} \left( X,\sum_{i=1}^k D_i \right)  = \inf_F \frac{A_X (F)}{\sum_{i=1}^k \mathrm{ord}_F( D_i )}.
\end{equation*}

Likewise, repeating the argument in \cite[Proof of Theorem 4.4]{bj20} immediately implies that the limit $\delta^{\text{coupled}} := \lim_{m \to \infty} \delta^{\text{coupled}}_{m}$ exists, and equals
\begin{equation*}
	\delta^{\text{coupled}} = \inf_F \frac{A_X (F)}{\sum_{i=1}^k S (L_i; F)}
\end{equation*}
where the infimum is over all prime divisors over $X$ and
\begin{equation*}
	S (L_i; F):=\frac{1}{\mathrm{Vol} (L_i)} \int_0^{+ \infty} \mathrm{Vol} (L_i -xF) dx .
\end{equation*}

It was proved by K.~Zhang \cite[Remark 5.3]{Zhang21} that $\delta^{\text{coupled}} >1$ implies the existence of coupled \ke metrics.

\bibliography{acbal.bib}


\end{document}